\documentclass[11pt]{article}

\usepackage[margin=1in]{geometry}
\usepackage{amsmath,amssymb,amsthm}
\usepackage{graphicx}
\usepackage{float}
\usepackage{array}
\usepackage{booktabs}
\usepackage{multirow}
\usepackage{makecell}
\usepackage{xcolor}
\usepackage{algorithm}
\usepackage{algpseudocode}
\usepackage{enumitem}
\usepackage{subcaption}
\usepackage{tabularx}
\usepackage{natbib}
\usepackage{hyperref}

\hypersetup{
    colorlinks=true,
    linkcolor=blue,
    citecolor=blue,
    urlcolor=blue
}
\numberwithin{equation}{section}

\theoremstyle{plain}
\newtheorem{theorem}{Theorem}[section]
\newtheorem{lemma}[theorem]{Lemma}
\newtheorem{proposition}[theorem]{Proposition}
\theoremstyle{definition}
\newtheorem{assumption}{Assumption}
\theoremstyle{remark}
\newtheorem{remark}[theorem]{Remark}

\title{Parameter-Free Adaptive Extragradient Methods for Generalized-Smooth Variational Inequalities\thanks{This research was supported by the National Natural Science Foundation of China (no.~12571329), the Shaanxi Provincial Program for Young Talents, the Fundamental Research Funds for the Central Universities (no.~QTZX26116), and the Xiaomi Young Scholar Program.}}

\author{
Zhaoyang Chu\thanks{Email: \texttt{chuzhaoyyang@163.com}}
\quad
Hongwei Liu\thanks{Email: \texttt{hwliuxidian@163.com}}
\quad
Yakui Huang\thanks{Corresponding author. Email: \texttt{huangyakui@xidian.edu.cn}}\\[1ex]
\small School of Mathematics and Statistics, Xidian University\\
\small Xi'an 710126, Shaanxi, China
}
\date{}

\begin{document}
\maketitle

\begin{abstract}
	Variational inequalities (VIs) provide a unified framework for  optimization, saddle-point,
	equilibrium, and learning problems, where operators need not be globally
	Lipschitz. The $\alpha$-symmetric generalized-smoothness condition allows local operator
	variation to grow with the operator norm,  but existing extragradient methods often
	require unknown problem-specific constants.	To address this issue, we propose PF-AEG, an explicit, line-search-free extragradient method whose stepsizes use	quantities generated during the iterations, without prior knowledge of these constants. Under
	$p$-quasi-sharpness, we prove full-sequence convergence 
	and last-iterate squared-distance rates: linear convergence for
	$0<p\le2$ and $\mathcal{O}(k^{-2/(p-2)})$ for $p>2$. The latter matches
	existing parameter-dependent rates under comparable assumptions, while
	the linear guarantee extends the previous $p=2$ result to 	$0<p\le2$.
	Building on the idea of PF-AEG, we develop AEG-VR, a variance-reduced adaptive
	extragradient method for finite-sum VIs. We prove almost-sure convergence and,
	under mild conditions, obtain analogous
	expected last-iterate rates, while retaining independence from the
	unknown problem-specific constants in stepsize selection.
	Experiments on synthetic generalized-smooth
	VIs and several learning problems show that the proposed methods are
	competitive with, and in several tested regimes outperform, existing
	baselines.
\end{abstract}

\par

\section{Introduction}
\label{sec:introduction}

Variational inequalities (VIs) provide a unified framework for optimization,
saddle-point problems, equilibrium models, and machine learning.
Given a nonempty closed convex set $U\subseteq\mathbb{R}^m$ and an operator
$F:U\to\mathbb{R}^m$, the VI problem seeks a point $u^*\in U$ such that
\begin{equation}\label{eq:VI}
	\langle F(u^*),u-u^*\rangle \ge 0,
	\qquad \forall u\in U.
\end{equation}
This formulation encompasses constrained convex optimization,
convex--concave saddle-point problems, complementarity problems, and a broad
class of equilibrium models \citep{facchinei2003vi,kinderlehrer1980vi}.
VIs have also become increasingly important in machine learning, particularly
for modeling and analyzing adversarial and multi-agent learning problems
\citep{gemp2018gans,gidel2019variational,kotsalis2022simple}.

\paragraph{Extragradient methods and local adaptation.}
For monotone VIs, Korpelevich's extragradient method and its extensions,
including Mirror-Prox and forward--backward--forward splitting, are among the
most widely used first-order methods
\citep{korpelevich1976extragradient,tseng2000modified,nemirovski2004prox}.
Their classical convergence analyses typically assume that the operator is
globally Lipschitz continuous:
\begin{equation}\label{eq:lipschitz-intro}
	\|F(u)-F(v)\|\le L\|u-v\|,
	\qquad \forall u,v\in U,
\end{equation}
for some $L>0$. Under this regularity, extragradient-type methods have well-established
convergence guarantees, including recent last-iterate results; for example,
\citet{gorbunov2022lastiterate} obtained an $\mathcal{O}(1/k)$ last-iterate rate for unconstrained monotone Lipschitz VIs.

A related line of work addresses stepsize selection without requiring a known
global Lipschitz constant. Representative examples include the adaptive golden-ratio algorithm of
\citet{malitsky2018graal}, self-adaptive projection and extragradient schemes
for Lipschitz-type VIs \citep{yang2019strong,liu2020weak}, and universal or locally adaptive
Mirror-Prox and golden-ratio frameworks \citep{bachlevy2019universal,enenguyen2022adaptive,alacaoglu2023golden}.
These methods demonstrate the value of local adaptation, but they primarily address unknown Lipschitz or related regularity rather than the operator-dependent behavior considered here.

\paragraph{Generalized smoothness beyond global Lipschitz continuity.} 
Global Lipschitz continuity can be restrictive when operator regularity varies
significantly across the domain. Empirical observations in deep-network
training, including LSTMs, show that local smoothness can vary substantially
along the optimization trajectory and correlate with the gradient norm,
motivating $(L_0,L_1)$-smoothness as a model for such non-Lipschitz behavior
\citep{zhang2019clipping,crawshaw2022robustness}. The
$\alpha$-symmetric generalized-smoothness framework of
\citet{chen2023generalized} further extends this model by allowing a broader
class of growth behaviors, covering, among other examples, high-order
polynomial and phase-retrieval objectives.

In the VI setting, the $\alpha$-symmetric generalized-smoothness condition
takes the following form: for some $L_0,L_1\ge0$ and
$\alpha\in(0,1]$,
\begin{equation}\label{eq:intro-alpha-smooth}
	\|F(u)-F(v)\|\le
	\left(L_0+L_1\max_{\theta\in[0,1]}
	\|F(\theta u+(1-\theta)v)\|^\alpha
	\right)\|u-v\|,
	\qquad u,v\in U.
\end{equation}
When $L_1=0$, this condition reduces to global Lipschitz continuity;
when $L_1>0$, the admissible local variation can increase with the operator
norm.

Extragradient-type methods under generalized smoothness have recently been
studied in both deterministic and stochastic settings.
\citet{vankov2024gsvimethods} analyzed projection, extragradient, and
Popov-type methods under generalized smoothness and relaxed monotonicity
conditions, while \citet{vankov2025gssvi} developed clipped stochastic
projection and extragradient methods under
$\alpha$-symmetric generalized smoothness and $p$-quasi-sharpness.
For the unconstrained case $U=\mathbb{R}^m$, the VI condition reduces to the
root-finding problem $F(u^*)=0$. In this special setting,
\citet{choudhury2025elllipschitz} studied extragradient methods under the
related $(L_0,L_1)$-Lipschitz condition. These results show that extragradient methods can remain effective beyond global Lipschitz continuity. However, their stepsize constructions typically
involve problem-specific constants, while stochastic schemes may
additionally rely on clipping. This motivates the question of whether an explicit adaptive stepsize can be constructed solely from quantities generated during the iterations, without
prior knowledge of the generalized-smoothness parameters.

\paragraph{Quasi-sharpness and last-iterate rates.}
Generalized smoothness controls local operator variation but does not by
itself quantify how rapidly the iterates approach the solution set.
To derive explicit last-iterate guarantees, we additionally consider the
$p$-quasi-sharpness condition
\begin{equation}\label{eq:intro-quasi-sharp}
	\langle F(u),u-u^*\rangle
	\ge\mu\operatorname{dist}^p(u,U^*),
	\qquad
	u\in U,\quad u^*\in U^*,
\end{equation}
where $\mu>0$, $p>0$, and $U^*$ denotes the solution set of the VI.
This condition provides an error-bound-type relation between the operator and
the distance to the solution set. In particular, $p=2$ corresponds to a quadratic error-bound-type geometry, whereas other values of $p$ describe different growth behavior near the
solution set. Error-bound and weak-sharpness conditions have been used to strengthen
qualitative convergence results into explicit last-iterate guarantees
\citep{malitsky2018graal,nan2023svrgeg}, and quasi-sharpness plays a central
role in recent analyses of generalized-smooth VIs
\citep{vankov2024gsvimethods,vankov2025gssvi}.


Under generalized smoothness and $p$-quasi-sharpness,
\citet{vankov2024gsvimethods} established
$\mathcal{O}(k^{-2/(p-2)})$ convergence for $p>2$ and linear convergence
for $p=2$ with parameter-dependent stepsizes. These results provide a natural benchmark for asking whether analogous last-iterate guarantees can be achieved without prior knowledge of the
underlying problem-specific parameters.

\paragraph{Finite-sum VIs and variance reduction.}
Another challenge arises in large-scale VIs with the finite-sum structure
\begin{equation}\label{finite-sum}
	F(u)=\frac{1}{N}\sum_{i=1}^N F_i(u),
\end{equation}
where evaluating the full operator $F$ can be prohibitively expensive when
$N$ is large.
Variance-reduction techniques have therefore been developed to reduce this
computational burden. \citet{palaniappan2016stochastic} developed variance-reduced methods with
linear convergence for finite-sum saddle-point problems, while
\citet{iusem2017variance} studied stochastic extragradient schemes with
variance reduction for stochastic VIs. More generally, \citet{alacaoglu2022stochvrvi} established a unified variance-reduced framework for extragradient and related methods for monotone
inclusions. Last-iterate convergence guarantees for SVRG-extragradient have also been
obtained under error-bound and weak sharpness conditions
\citep{nan2023svrgeg}.

These variance-reduced analyses are primarily developed under standard
Lipschitz-type regularity. Under generalized smoothness, however, the stochastic estimation error must be controlled together with operator-dependent local variation, while
stepsize selection should ideally remain independent of unknown
problem-specific constants. Combining these requirements in a variance-reduced extragradient framework therefore presents an additional challenge in the finite-sum setting.

We address the above issues through a common adaptive extragradient framework.
For deterministic VIs, we propose PF-AEG, an explicit, line-search-free
adaptive extragradient method whose stepsizes are computed solely from
quantities generated during the iterations, without prior knowledge of the
generalized-smoothness parameters $L_0,L_1,\alpha$ or the
quasi-sharpness parameters $p,\mu$.
Building on the idea of PF-AEG, we develop AEG-VR, a variance-reduced adaptive
extragradient method for finite-sum VIs.

\paragraph{Main contributions.}
Our main contributions are summarized as follows.
\begin{enumerate}[label=(\roman*)]
	\item \textbf{Parameter-free adaptive extragradient method.}
	We propose PF-AEG for VIs with $\alpha$-symmetric generalized smooth operators.
	Its explicit stepsize rule is computed solely from quantities generated by the
	extragradient iterations and requires neither a line search nor prior knowledge
	of $L_0,L_1,\alpha,p$, or $\mu$. A key technical ingredient is to show
	that the resulting adaptive stepsizes remain bounded away from zero despite
	the absence of these problem-specific constants. To the best of our knowledge,
	PF-AEG is the first explicit, line-search-free extragradient method for this
	class whose stepsize selection is independent of the generalized-smoothness
	and quasi-sharpness parameters.

	\item \textbf{Full-sequence convergence and last-iterate rates.}
 Under $p$-quasi-sharpness, we prove convergence of the iteration sequence generated by PF-AEG and	establish last-iterate rates for the squared distance to the solution set:
 linear convergence for $0<p\le2$ and
	$\mathcal{O}(k^{-2/(p-2)})$ for $p>2$.
	The $p>2$ rate matches existing parameter-dependent results under
	comparable assumptions, while for $0<p\le2$, the linear rate
	extends the $p=2$ result of \citet{vankov2024gsvimethods}.

	\item \textbf{Variance-reduced adaptive extragradient method.}
	For the finite-sum model \eqref{finite-sum}, we develop AEG-VR by combining
	the adaptive principle of PF-AEG with an SVRG-type variance-reduction
	mechanism and a randomized snapshot update.
	We prove almost-sure convergence and, under the uniform tail conditions
	used in the rate analysis, obtain analogous expected last-iterate rates:
	linear convergence for $0<p\le2$ and
	$\mathcal{O}(k^{-2/(p-2)})$ for $p>2$.
	The stepsize selection retains independence from the unknown
	problem-specific constants.

	\item \textbf{Numerical evaluation.}
	We evaluate PF-AEG and AEG-VR on a synthetic generalized-smooth VI and
	several learning problems, including logistic regression, phase retrieval,
	and distributionally robust optimization. Across the tested instances, PF-AEG often reduces the
	reported error or objective value faster per iteration than the adaptive,
	clipped, normalized, and backtracking baselines. In several finite-sum
	experiments, AEG-VR also outperforms the tested stochastic and
	variance-reduced baselines at matched reported sample budgets.
\end{enumerate}

Tables~\ref{tab:deterministic_generalized_stepsizes} and
\ref{tab:stochastic_generalized_stepsizes} summarize the closest theoretical
comparisons. The novel extragradient result of \citet{choudhury2025elllipschitz} concerns the unconstrained root-finding case $F(u^*)=0$, which corresponds to the VI problem with
$U=\mathbb{R}^m$. It is included to compare generalized-smooth stepsize
constructions, rather than as a direct comparison under identical problem
settings.

\begin{table}[H]
	\centering
	\caption{Comparison of stepsize rules and reported convergence guarantees for
		deterministic extragradient-type methods beyond global Lipschitz continuity.}
	\label{tab:deterministic_generalized_stepsizes}
	\renewcommand{\arraystretch}{2.2}
	\setlength{\tabcolsep}{2pt}
	\arrayrulewidth=0.4pt
	\resizebox{\textwidth}{!}{%
		\scriptsize
		\begin{tabular}{
				|>{\centering\arraybackslash}m{2.1cm}|
				>{\centering\arraybackslash}m{0.7cm}|
				>{\tiny\centering\arraybackslash}m{10.9cm}|
				>{\tiny\centering\arraybackslash}m{2.7cm}|
			}
			\hline
			\textbf{Algorithm}
			& \boldmath$\alpha$
			& \textbf{Stepsizes rule}
			& \textbf{Convergence}
			\\
			\hline
			
			\multirow{2}{2.2cm}{\centering\textbf{Extragradient}\newline \cite{vankov2024gsvimethods}}
			&
			$1$
			&
			$\displaystyle \lambda_k =
			\min\left\{\frac{1}{4\mu},
			\frac{1}{2\sqrt{2}eL_0},
			\frac{1}{2\sqrt{2}eL_1\|F(z_k)\|}
			\right\}
			$
			&
			$\displaystyle
			\begin{gathered}
				\mathcal{O}\!\left(k^{-2/(p-2)}\right),~ p>2;\\[0.25pt]
				\text{linear},~ p=2
			\end{gathered}
			$
			\\
			\cline{2-4}
			&
			$(0,1)$
			&
			$\displaystyle \lambda_k =
			\min\left\{\frac{1}{4\mu},
			\frac{1}{3\sqrt{2}K_0},
			\frac{1}{\|F(z_k)\|},
			\frac{1}{3\sqrt{2}K_1\|F(z_k)\|^\alpha},
			\frac{1}{3\sqrt{2}K_2}
			\right\}
			$
			&
			$\displaystyle
			\begin{gathered}
				\mathcal{O}\!\left(k^{-2/(p-2)}\right),~ p>2;\\[0.25pt]
				\text{linear},~ p=2
			\end{gathered}
			$
			\\
			\hline
			
			\multirow{2}{2.2cm}{\centering\textbf{Popov}\newline
			\cite{vankov2024gsvimethods}}
			&
			$1$
			&
			$\displaystyle
			\lambda_k =
			\min\left\{
			\frac{1}{4\mu},
			\frac{1}{\|F(z_{k-1})\|},
			\frac{1}{
				2\sqrt{2}
				(L_0+L_1\|F(z_{k-1})\|)
				e^{L_1\|t_k-z_{k-1}\|+1}}
			\right\}
			$
			&
			$\displaystyle
			\begin{gathered}
				\mathcal{O}\!\left(k^{-2/(p-2)}\right),~ p>2;\\[0.25pt]
				\text{linear},~ p=2
			\end{gathered}
			$
			\\
			\cline{2-4}
			&
			$(0,1)$
			&
			$\displaystyle
			\begin{gathered}
				\lambda_k = \min\left\{
				\frac{1}{\|F(z_{k-1})\|},
				\frac{1}{6\sqrt{2}K_0},
				\frac{1}{6\sqrt{2}K_1\|F(z_{k-1})\|^\alpha}, 
				\frac{1}{
					6\sqrt{2}K_2
					(\|t_k-z_{k-1}\|+1)^{\alpha/(1-\alpha)}
				},
				\frac{1}{4\mu}
				\right\}
			\end{gathered}
			$
			&
			$\displaystyle
			\begin{gathered}
				\mathcal{O}\!\left(k^{-2/(p-2)}\right),~ p>2;\\[0.25pt]
				\text{linear},~ p=2
			\end{gathered}
			$
			\\
			\hline
			
			\multirow{2}{2.2cm}{\centering\textbf{Novel extragradient}\newline \cite{choudhury2025elllipschitz}}
			&
			$1$
			&
			$\displaystyle
			\lambda_k =
			\frac{0.21}{L_0+L_1\|F(z_k)\|}
			$
			&
			sublinear 
			\\
			\cline{2-4}
			&
			$(0,1)$
			&
			$\displaystyle
			\lambda_k =
			\frac{0.61}{
				2K_0+
				(2K_1+2^{1-\alpha}K_2^{1-\alpha})
				\|F(z_k)\|^\alpha}
			$
			&
			sublinear
			\\
			\hline
			
			\multirow{1}{2.2cm}{\centering\textbf{PF-AEG (Proposed)}}
			&
			$(0,1]$
			&
			$\displaystyle
			\lambda_k =
			\min\left\{(1+\eta_{k-1})\lambda_{k-1},
			\frac{\|z_{k-1}-t_{k-1}\|^2}
			{2\lambda_{k-1}\|F(z_{k-1})-F(t_{k-1})\|^2},
			\frac{1}{\|F(z_k)\|^2}
			\right\},~\eta_{k-1}= \frac{1}{k^{2}}
			$
			&
			$\displaystyle
			\begin{gathered}
				\mathcal{O}\!\left(k^{-2/(p-2)}\right),~ p>2;\\[0.25pt]
				\text{\bfseries linear},~ \boldsymbol{0<p\leq 2}
			\end{gathered}
			$
			\\
			\hline
		\end{tabular}%
	}
	
	\vspace{1mm}
	\begin{minipage}{0.98\textwidth}
		\footnotesize	
		\emph{Notes.}
		Here, $\mu$ denotes the $p$-quasi-sharpness constant, while
		$L_0,L_1,K_0,K_1,K_2$ denote the regularity constants appearing in the
		corresponding generalized-smoothness analyses. The Extragradient and Popov methods of
		\cite{vankov2024gsvimethods} and PF-AEG are analyzed for VIs under
		$p$-quasi-sharpness. Novel extragradient in \cite{choudhury2025elllipschitz} considers the
		unconstrained root-finding problem $F(u^*)=0$, and assumes monotonicity rather than	$p$-quasi-sharpness.
	\end{minipage}
\end{table}

\begin{table}[H]
	\centering
	\caption{Comparison of stepsize rules and reported convergence guarantees for
		stochastic and finite-sum variance-reduced extragradient methods under
		$\alpha$-symmetric generalized smoothness and $p$-quasi-sharpness.}
	\label{tab:stochastic_generalized_stepsizes}
	\renewcommand{\arraystretch}{1.35}
	\setlength{\tabcolsep}{4pt}
	\arrayrulewidth=0.4pt
	\scriptsize
	\resizebox{\textwidth}{!}{%
		\begin{tabular}{
				|>{\centering\arraybackslash}m{2.1cm}|
				>{\centering\arraybackslash}m{0.7cm}|
				>{\raggedright\centering\arraybackslash}m{13.0cm}|
				>{\centering\arraybackslash}m{4.4cm}|
			}
			\hline
			\textbf{Algorithm}
			& \boldmath$\alpha$
			& \textbf{Stepsizes rule}
			& \textbf{Convergence}
			\\
			\hline
			
			\textbf{Clipped stochastic extragradient} \cite{vankov2025gssvi}
			&
			$(0,\frac{1}{2}]$
			&
			$\lambda_k=\gamma_k\min\left\{1,\frac{1}{\|F(z_k,\xi_k^1)\|}\right\}$
			&
			$\displaystyle
			\begin{gathered}
				\mathcal{O}(k^{-2(1-q)/p}),\;\; p>2,\;\; q\in(\frac{1}{2},1);\\[0.25pt]
				\mathcal{O}\!\left(1/k\right),\quad p=2, 
			\end{gathered}
			$
			\\
			\hline
			
			\textbf{AEG-VR (Proposed)}
			&
			$(0,1]$
			&
			$\displaystyle
			\lambda_{k} =
			\min\left\{
			(1+\eta_{k-1})\lambda_{k-1},
			\frac{\beta \|t^{k-1/2}-s^{k-1}\|}{\|F_{\xi_{k-1}}(t^{k-1/2}) - F_{\xi_{k-1}}(s^{k-1})\|},
			\frac{1}{\max_{i\in[N]}\|F_i(s^k)\|^2}
			\right\},
			\quad \eta_{k-1}= \frac{1}{k^{2}}
			$
			&
			$\displaystyle
			\begin{gathered}
				\mathcal{O}\!\left(k^{-2/(p-2)}\right),\quad p>2;\\[0.25pt]
				\text{\bfseries linear},\quad \boldsymbol{0<p\leq 2}
			\end{gathered}
			$
			\\
			\hline
		\end{tabular}%
	}
	
	\vspace{1mm}
	\begin{minipage}{0.98\textwidth}
		\footnotesize
		\emph{Notes.}
		For clipped stochastic extragradient, $F(z_k,\xi_k^1)$ denotes a stochastic oracle evaluation of $F$, and $\{\gamma_k\}$ satisfies	$\sum_{k=0}^{\infty}\gamma_k^2<\infty$.
		For AEG-VR, $[N]:=\{1,\ldots,N\}$, $\xi\in[N]$ denotes a sampled index,
		and $F_{\xi_{k-1}}$ is the corresponding component operator.
		Moreover, $\max_{i\in[N]}\|F_i(s^k)\|$ denotes the maximum
		component-operator norm at $s^k$, and $\beta\in(0,\frac12)$ is a user-specified algorithmic parameter whose	selection is independent of $L_0,L_1,\alpha,p$, and $\mu$. The displayed AEG-VR rates are expected last-iterate rates under the uniform
		tail conditions used in the rate analysis.
	\end{minipage}
\end{table}

\paragraph{Organization.}
The rest of this paper is organized as follows. Section~\ref{sec:preliminaries} introduces notations, assumptions, and auxiliary facts.  Section~\ref{sec:pf-aeg} presents PF-AEG and its convergence analysis. Section~\ref{sec:aeg-vr} develops AEG-VR for finite-sum VIs.
Section~\ref{sec:experiments} reports the numerical experiments, and Section~\ref{sec:conclusion} concludes the paper.

\section{Preliminaries}\label{sec:preliminaries}

This section collects the notation, assumptions, and auxiliary results used in
the convergence analysis.

\subsection{Notation and Assumptions}

Throughout the paper, $\|\cdot\|$ denotes the Euclidean norm and
$\langle\cdot,\cdot\rangle$ the standard inner product in
$\mathbb{R}^m$. For a nonempty closed convex set
$C\subseteq\mathbb{R}^m$, let $P_C$ denote the Euclidean projection onto
$C$:
\begin{equation*}
	P_C(x):=\arg\min_{u\in C}\|u-x\|.
\end{equation*}
For a nonempty set $S\subseteq\mathbb{R}^m$, define
\begin{equation*}
	\operatorname{dist}(x,S):=\inf_{s\in S}\|x-s\|.
\end{equation*}
We denote by $U^*$ the solution set of the VI \eqref{eq:VI}.

We make the following assumptions.

\begin{assumption}\label{ass1}
	The set $U\subseteq\mathbb{R}^m$ is nonempty, closed, and convex.
\end{assumption}

\begin{assumption}[$\alpha$-symmetric generalized smoothness]\label{ass2}
	The operator $F:U\to\mathbb{R}^m$ is continuous, and
	\eqref{eq:intro-alpha-smooth} holds for some
	$L_0,L_1\ge0$ and $\alpha\in(0,1]$.
\end{assumption}

\begin{assumption}[$p$-quasi-sharpness]\label{ass3}
	The solution set $U^*$ is nonempty, and
	\eqref{eq:intro-quasi-sharp} holds for some
	$p>0$ and $\mu>0$.
\end{assumption}

Under Assumptions~\ref{ass1}--\ref{ass2}, the solution set $U^*$ is closed.
Together with Assumption~\ref{ass3}, this implies that, for every
$x\in\mathbb{R}^m$, there exists at least one $x^*\in U^*$ such that
\begin{equation*}
	\|x-x^*\|=\operatorname{dist}(x,U^*).
\end{equation*}
Such a nearest point need not be unique, since $U^*$ is not assumed to be
convex.

\subsection{Auxiliary Results}

We first recall an equivalent characterization of
$\alpha$-symmetric generalized smoothness.
\begin{proposition}[\cite{chen2023generalized}, Proposition 1]\label{pro3.1}
	Assume that $U\subseteq\mathbb{R}^m$ is a nonempty set, and $F: U\to\mathbb{R}^m$ is an operator. Then,
	\begin{enumerate}
	\item[$\mathbf{(i)}$] $F$ is $\alpha$-symmetric smooth with $\alpha\in(0,1)$ and constants $L_0, L_1\geq0$, if and only if,	for all $y,y'\in U$,
	\begin{equation*}
		\|F(y)-F(y')\|\leq\|y-y'\|\left(K_0+K_1\|F(y')\|^\alpha
		+K_2\|y-y'\|^{\frac{\alpha}{1 - \alpha}}\right),
	\end{equation*}
	where ${K}_{0} = {L}_{0}\left( 2^{\frac{\alpha^2}{1 - \alpha}} + 1\right)$, ${K}_{1}=L_1{2}^{\frac{\alpha^2}{1 - \alpha}}$, and ${K}_2= {L}_1^{\frac{1}{1 - \alpha}}{2}^{\frac{\alpha^2}{1 - \alpha}}3^{\alpha}(1 - \alpha)^{\frac{\alpha}{1 - \alpha}}$.
	
	\item[$\mathbf{(ii)}$] $F$ is $\alpha$-symmetric smooth with $\alpha=1$ and constants $L_0, L_1\geq0$, if and only if,  for all $y,y'\in U$,
	\begin{equation*}
		\|F(y)-F(y')\|\leq\|y-y'\|\bigl(L_0+L_1\|F(y')\|\bigr)\exp\bigl(L_1\|y-y'\|\bigr).
	\end{equation*}
\end{enumerate}
\end{proposition}

The following properties of Euclidean projections will be used in the
analysis.

\begin{lemma}[\citealp{facchinei2003vi}, Theorem~1.5.5 and Lemma~12.1.13]
	\label{lem21}
	Let $U\subseteq\mathbb{R}^m$ be nonempty, closed, and convex. Then, for
	every $u\in U$ and $v\in\mathbb{R}^m$,
	\begin{equation}\label{2.1}
		\langle v-P_U(v),u-P_U(v)\rangle\le0,
	\end{equation}
	and
	\begin{equation}\label{2.2}
		\|u-P_U(v)\|^2\leq\|u-v\|^2-\|v-P_U(v)\|^2.
	\end{equation}
	Moreover, $P_U$ is nonexpansive:
	\begin{equation}\label{2.3}
		\|P_U(x)-P_U(y)\|\le\|x-y\|,
		\qquad x,y\in\mathbb{R}^m.
	\end{equation}
\end{lemma}

\begin{lemma}[\citealp{bauschke2017convex}, Lemma~5.31]
	\label{lem23}
	Let $\{v_k\}$, $\{w_k\}$, $\{a_k\}$, and $\{b_k\}$ be
	nonnegative scalar sequences satisfying
	\begin{equation*}
		v_{k+1}\le(1+a_k)v_k-w_k+b_k,
		\qquad k\ge0.
	\end{equation*}
	If
	\begin{equation*}
		\sum_{k=0}^{\infty}a_k<\infty,
		\qquad
		\sum_{k=0}^{\infty}b_k<\infty,
	\end{equation*}
	then $\{v_k\}$ converges to some $v\ge0$, and
	\begin{equation*}
		\sum_{k=0}^{\infty}w_k<\infty.
	\end{equation*}
\end{lemma}

\begin{lemma}[\citealp{bauschke2017convex}, Lemma~5.5]
	\label{lem:Fejer}
	Let $S\subseteq\mathbb{R}^m$ be nonempty, and let
	$\{x_k\}\subseteq\mathbb{R}^m$ be Fej\'er monotone with respect to $S$,
	that is,
	\begin{equation*}
		\|x_{k+1}-x\|\le\|x_k-x\|,
		\qquad
		\forall x\in S,\quad k\ge0.
	\end{equation*}
	If every accumulation point of $\{x_k\}$ belongs to $S$, then
	$\{x_k\}$ converges to a point in $S$.
\end{lemma}

\begin{lemma}[Robbins--Siegmund; see \citealp{polyak1987intro}]
	\label{lem:rs}
	Let $\{v_k\}$, $\{z_k\}$, $\{a_k\}$, and $\{b_k\}$ be
	nonnegative random variables adapted to a filtration
	$\{\mathcal F_k\}$. Suppose that
	\begin{equation*}
		\mathbb E[v_{k+1}\mid\mathcal F_k]
		\le(1+a_k)v_k-z_k+b_k,
		\qquad k\ge0,
	\end{equation*}
	almost surely, and that
	\begin{equation*}
		\sum_{k=0}^{\infty}a_k<\infty,
		\qquad
		\sum_{k=0}^{\infty}b_k<\infty
		\qquad\text{almost surely}.
	\end{equation*}
	Then $v_k$ converges almost surely to a finite nonnegative random variable,
	and
	\begin{equation*}
		\sum_{k=0}^{\infty}z_k<\infty
		\qquad\text{almost surely}.
	\end{equation*}
\end{lemma}

\begin{lemma}[\citealp{polyak1987intro}, Chapter~2, Lemma~6]
	\label{lem25}
	Let $\{v_k\}$ and $\{a_k\}$ be nonnegative scalar sequences, and let
	$q>0$. Suppose that
	\begin{equation*}
		v_{k+1}\le v_k-a_kv_k^{1+q},
		\qquad k\ge0.
	\end{equation*}
	Then
	\begin{equation*}
		v_k\le\frac{v_0}{\left(1+qv_0^q\sum_{i=0}^{k-1}a_i\right)^{1/q}}.
	\end{equation*}
	In particular, if $a_k=a>0$ for all $k$, then
	\begin{equation*}
		v_k\le\frac{v_0}{\left(1+aqv_0^q k\right)^{1/q}}.
	\end{equation*}
\end{lemma}

\section{Parameter-Free Adaptive Extragradient Method}\label{sec:pf-aeg}

In this section, we introduce PF-AEG, a parameter-free adaptive extragradient
method for VIs with $\alpha$-symmetric operators. The method
uses an explicit stepsize rule based solely on quantities generated by the
iterations and requires neither prior knowledge of the generalized-smoothness
parameters $L_0,L_1,\alpha$ or the quasi-sharpness parameters $p,\mu$, nor
a line search.

\subsection{Algorithmic Framework}
For an $\alpha$-symmetric operator, the regularity depends on its magnitude and can vary along the iterates. This motivates adapting the stepsize using the observed operator change
$\|F(z_k)-F(t_k)\|$ relative to the iterate displacement
$\|z_k-t_k\|$, together with the operator magnitude. Moreover, a stepsize
reduced due to a large local operator change may become unnecessarily
conservative at subsequent iterates. The update therefore allows the stepsize
to recover gradually  rather than remain constrained by earlier unfavorable local behavior.

\begin{algorithm}[ht]
	\caption{Parameter-Free Adaptive Extragradient Method (PF-AEG)}
	\label{alg3}
	
	\textbf{Initialization:}
	Choose $z_0\in U$, $\lambda_0>0$, and set $k=0$.
	
	\vspace{1mm}
	\textit{Step 1:}
	Compute
	\begin{equation*}
		t_k=P_U\!\left(z_k-\lambda_k F(z_k)\right),
	\end{equation*}
	\begin{equation*}
		z_{k+1}=P_U\!\left(z_k-\lambda_k F(t_k)\right).
	\end{equation*}
	
	\textit{Step 2:}
	Set
	\begin{equation*}
		\eta_k=\frac{1}{(k+1)^2},
	\end{equation*}
	and update
	\begin{equation}\label{eq:pf-aeg-stepsize}
		\lambda_{k+1}=\min\left\{
		(1+\eta_k)\lambda_k,\,
		\frac{\|z_k-t_k\|^2}
		{2\lambda_k\|F(z_k)-F(t_k)\|^2},\,
		\frac{1}{\|F(z_{k+1})\|^2}
		\right\}.
	\end{equation}
	
	\textit{Step 3:}
	Set $k\gets k+1$ and return to Step~1.
\end{algorithm}

PF-AEG incorporates these considerations into a three-way adaptive stepsize
rule, combined with the extragradient step. The complete algorithm is given in Algorithm~\ref{alg3}. Throughout, a candidate stepsize with a zero denominator is interpreted as $+\infty$.

The three terms in \eqref{eq:pf-aeg-stepsize} play complementary
roles. The first term, $(1+\eta_k)\lambda_k$, controls the increase of the stepsize between successive iterations while allowing gradual recovery after a previous reduction. Thus,
the method need not remain unnecessarily conservative after passing through a
region of large local operator variation.

The second term, $\frac{\|z_k-t_k\|^2}{2\lambda_k\|F(z_k)-F(t_k)\|^2}$, adapts to the operator variation observed in the current
extragradient step. Since $\lambda_{k+1}$ is no larger than this term,
we have
\begin{equation}\label{eq:local-difference-control}
	\lambda_{k+1}\lambda_k\|F(z_k)-F(t_k)\|^2
	\le\frac{1}{2}\|z_k-t_k\|^2.
\end{equation}
This inequality provides the key control of the operator-difference term in
the descent analysis and is enforced directly from the computed iterates and
operator evaluations.

The third term in \eqref{eq:pf-aeg-stepsize} controls the
operator-norm-dependent growth permitted by generalized smoothness. In fact,
for every $k\ge0$,
$\lambda_{k+1}\le 1/\|F(z_{k+1})\|^2$, and hence
\begin{equation}\label{eq:operator-norm-control}
	\lambda_k\|F(z_k)\|^2\le1,
	\qquad k\ge1.
\end{equation}
Consequently, for every $\alpha\in(0,1]$,
\begin{equation*}
	\lambda_k\|F(z_k)\|^\alpha=
	\lambda_k^{1-\alpha/2}
	\bigl(\lambda_k\|F(z_k)\|^2\bigr)^{\alpha/2}
	\le\bar{\lambda}^{\,1-\alpha/2},
	\qquad k\ge1,
\end{equation*}
where $\bar{\lambda}$ is the uniform upper bound established in
Lemma~\ref{lem31}. Thus, after the first iteration, the quadratic
normalization controls $\lambda_k\|F(z_k)\|^\alpha$ for every
$\alpha\in(0,1]$ without requiring the value of $\alpha$.

In summary, the three components of the stepsize rule provide, respectively,
controlled stepsize recovery, adaptation to local operator variation, and
control of operator-norm growth. Since all quantities in
\eqref{eq:pf-aeg-stepsize} are generated during the iterations, PF-AEG
requires no prior knowledge of $L_0,L_1,\alpha,p$, or $\mu$, and performs
no line search.

\subsection{Convergence Analysis}
\label{subsec:pf-aeg-convergence}

We now analyze the convergence of PF-AEG under the $\alpha$-symmetric smoothness and $p$-quasi-sharpness conditions. We first establish basic properties of the adaptive stepsizes and a one-step extragradient estimate. These results are then used to prove a quasi-Fej\'er descent inequality, a positive lower bound on the stepsizes, full-sequence convergence, and explicit last-iterate rates.

\begin{lemma}\label{lem31}
	Let $\{\lambda_k\}$ be the stepsize sequence generated by Algorithm~\ref{alg3}. Then the following statements hold.
	\begin{enumerate}
		\item[$\mathbf{(i)}$] The sequence is uniformly bounded above:
		\begin{equation}\label{eq:lambda-upper}
			0<\lambda_k\le \bar\lambda
			:=\lambda_0\exp\!\left(\sum_{j=0}^{\infty}\eta_j\right),
			\qquad k\ge0.
		\end{equation}
		
		\item[$\mathbf{(ii)}$] The stepsize sequence converges:
		\begin{equation}\label{eq:lambda-limit}
			\lambda_k\to\lambda_\infty
			\qquad\text{for some }\lambda_\infty\ge0.
		\end{equation}
		
		\item[$\mathbf{(iii)}$] The positive and negative variations are summable:
		\begin{equation}\label{eq:lambda-variation}
			\sum_{k=0}^{\infty}(\lambda_{k+1}-\lambda_k)_+<\infty,
			\qquad
			\sum_{k=0}^{\infty}(\lambda_{k+1}-\lambda_k)_-<\infty,
		\end{equation}
		where
		\begin{equation*}
			(a)_+:=\max\{a,0\},
			\qquad
			(a)_-:=\max\{-a,0\}.
		\end{equation*}
		Consequently,
		\begin{equation}\label{eq:lambda-total-variation}
			\sum_{k=0}^{\infty}|\lambda_{k+1}-\lambda_k|<\infty.
		\end{equation}
	\end{enumerate}
\end{lemma}

\begin{proof}
	By the first term in the stepsize rule \eqref{eq:pf-aeg-stepsize}, $\lambda_{k+1}\le(1+\eta_k)\lambda_k$. 	Iterating this inequality and using $\log(1+x)\le x$ for $x\ge0$, we obtain
	\begin{align*}
		\lambda_{k+1}
		&\le \lambda_0\prod_{j=0}^{k}(1+\eta_j)
		\le \lambda_0\exp\!\left(\sum_{j=0}^{k}\eta_j\right)
		\le \lambda_0\exp\!\left(\sum_{j=0}^{\infty}\eta_j\right)
		=\bar\lambda,
	\end{align*}
	which proves $\mathbf{(i)}$.
	
	We next prove $\mathbf{(iii)}$. The stepsize rule \eqref{eq:pf-aeg-stepsize} gives
	$\lambda_{k+1}-\lambda_k\le\eta_k\lambda_k\le\bar\lambda\eta_k$. Hence
	\begin{equation*}
		(\lambda_{k+1}-\lambda_k)_+
		\le \bar\lambda\eta_k.
	\end{equation*}
	Since $\sum_{k=0}^{\infty}\eta_k<\infty$, it follows that
	\begin{equation*}
		\sum_{k=0}^{\infty}(\lambda_{k+1}-\lambda_k)_+<\infty.
	\end{equation*}
	Moreover, for every $N\ge0$,
	\begin{align*}
		\lambda_{N+1}-\lambda_0
		&=\sum_{k=0}^{N}(\lambda_{k+1}-\lambda_k)_+
		-\sum_{k=0}^{N}(\lambda_{k+1}-\lambda_k)_-.
	\end{align*}
	Since $\lambda_{N+1}\ge0$,
	\begin{equation*}
		\sum_{k=0}^{N}(\lambda_{k+1}-\lambda_k)_-
		\le
		\lambda_0+\sum_{k=0}^{N}(\lambda_{k+1}-\lambda_k)_+.
	\end{equation*}
	Letting $N\to\infty$ proves the second assertion in \eqref{eq:lambda-variation}, and therefore \eqref{eq:lambda-total-variation} follows.
	
	Finally, \eqref{eq:lambda-total-variation} implies that $\{\lambda_k\}$ is a Cauchy sequence in $\mathbb{R}$. Thus it converges to a finite limit $\lambda_\infty\ge0$, proving $\mathbf{(ii)}$.
\end{proof}

\begin{remark}
	It is easy to see that Lemma~\ref{lem31} applies  whenever $\{\eta_k\}$ is summable, which in turn guarantees the full-sequence convergence and explicit last-iterate rates established below. To simplify the analysis, however, we set $\eta_k=\frac{1}{(k+1)^2}$ in PF-AEG, which also performs well in our experiments.
\end{remark}

\begin{lemma}\label{lem32}
	Let $\{t_k\}$ and $\{z_k\}$ be generated by Algorithm~\ref{alg3}. Then, for every $y\in U$ and every $k\ge0$,
	\begin{equation}\label{4.1}
		\|z_{k+1}-y\|^2\le
		\|z_k-y\|^2-\|z_k-t_k\|^2
		-2\lambda_k\langle F(t_k),t_k-y\rangle
		+\lambda_k^2\|F(z_k)-F(t_k)\|^2.
	\end{equation}
\end{lemma}

\begin{proof}
	Since
	\begin{equation*}
		z_{k+1}=P_U\bigl(z_k-\lambda_kF(t_k)\bigr),
	\end{equation*}
	the projection inequality \eqref{2.2} gives, for every $y\in U$,
	\begin{align}
		\|z_{k+1}-y\|^2
		&\le\|z_k-\lambda_kF(t_k)-y\|^2
		-\|z_k-\lambda_kF(t_k)-z_{k+1}\|^2 \notag\\
		&=\|z_k-y\|^2-\|z_{k+1}-z_k\|^2
		+2\lambda_k\langle F(t_k),y-z_{k+1}\rangle.
		\label{4.3}
	\end{align}
	Using
	\begin{equation*}
		\|z_{k+1}-z_k\|^2=\|z_{k+1}-t_k\|^2
		+\|z_k-t_k\|^2-2\langle z_{k+1}-t_k,z_k-t_k\rangle,
	\end{equation*}
	we can rewrite \eqref{4.3} as
	\begin{align}
		\|z_{k+1}-y\|^2
		&\le\|z_k-y\|^2-\|z_{k+1}-t_k\|^2-\|z_k-t_k\|^2
		-2\lambda_k\langle F(t_k),t_k-y\rangle \notag\\
		&\quad
		+2\langle z_{k+1}-t_k,
		z_k-\lambda_kF(z_k)-t_k\rangle
		+2\lambda_k\langle F(z_k)-F(t_k),z_{k+1}-t_k\rangle.
		\label{4.5}
	\end{align}
	Because $t_k=P_U\bigl(z_k-\lambda_kF(z_k)\bigr)$,
	the projection optimality condition \eqref{2.1}, with $z_{k+1}\in U$, implies
	\begin{equation*}
		\langle z_{k+1}-t_k,z_k-\lambda_kF(z_k)-t_k\rangle\le0.
	\end{equation*}
	In addition, the Cauchy--Schwarz and Young inequalities yield
	\begin{equation*}
		2\lambda_k\langle F(z_k)-F(t_k),z_{k+1}-t_k\rangle
		\le
		\lambda_k^2\|F(z_k)-F(t_k)\|^2
		+\|z_{k+1}-t_k\|^2.
	\end{equation*}
	Substituting these two estimates into \eqref{4.5} gives \eqref{4.1}.
\end{proof}

\begin{theorem}\label{th61}
	Suppose that Assumptions~\ref{ass1}--\ref{ass3} hold, and let
	$\{t_k\}$ and $\{z_k\}$ be generated by Algorithm~\ref{alg3}. Then:
	\begin{enumerate}
		\item[$\mathbf{(i)}$]
		For every $t^*\in U^*$ and every $k\ge1$,
		\begin{equation}\label{eq:descent}
			\|z_{k+1}-t^*\|^2\le\|z_k-t^*\|^2-\frac12\|z_k-t_k\|^2
			-2\mu\lambda_k\operatorname{dist}^p(t_k,U^*)
			+M(\lambda_k-\lambda_{k+1})_+,
		\end{equation}
		where
		\begin{equation*}
			M=
			\begin{cases}
				\displaystyle
				\left(
				\bar{\lambda}K_0
				+\bar{\lambda}^{\frac{2-\alpha}{2}}K_1
				+\bar{\lambda}^{\frac{2-\alpha}{2(1-\alpha)}}K_2
				\right)^2,
				& \alpha\in(0,1),\\[3mm]
				\displaystyle
				\left(
				\bar{\lambda}L_0+\sqrt{\bar{\lambda}}\,L_1
				\right)^2
				e^{2L_1\sqrt{\bar{\lambda}}},
				& \alpha=1.
			\end{cases}
		\end{equation*}
		
		\item[$\mathbf{(ii)}$]
		There exists $C>0$ such that $\|F(z_k)\|\le C$ for all $k\ge0$.
		Moreover,
		\begin{equation}\label{eq:lambda-lower}
			\lambda_k\ge\underline{\lambda}>0,
			\qquad k\ge0,
		\end{equation}
		where
		\begin{equation*}
			\underline{\lambda}
			=
			\begin{cases}
				\displaystyle
				\min\left\{
				\lambda_0,\,
				\lambda_1,\,
				\frac{1}{
					2\bar{\lambda}
					\left(
					K_0+K_1C^\alpha
					+K_2\bar{\lambda}^{\frac{\alpha}{2(1-\alpha)}}
					\right)^2
				},\,
				\frac{1}{C^2}
				\right\},
				& \alpha\in(0,1),\\[4mm]
				\displaystyle
				\min\left\{
				\lambda_0,\,
				\lambda_1,\,
				\frac{1}{
					2\bar{\lambda}(L_0+L_1C)^2
					e^{2L_1\sqrt{\bar{\lambda}}}
				},\,
				\frac{1}{C^2}
				\right\},
				& \alpha=1.
			\end{cases}
		\end{equation*}
		
		\item[$\mathbf{(iii)}$]
		Both $\{z_k\}$ and $\{t_k\}$ converge to a common solution
		$\bar t\in U^*$.
	\end{enumerate}
\end{theorem}
\begin{proof}
	\textbf{(i)}
	By \eqref{eq:operator-norm-control} and the nonexpansiveness of $P_U$, for every $k\ge1$,
	\begin{align}
		\|z_k-t_k\|
		&=\left\|P_U(z_k)-P_U\bigl(z_k-\lambda_kF(z_k)\bigr)\right\| \notag\\
		&\le
		\lambda_k\|F(z_k)\|=
		\sqrt{\lambda_k}
		\bigl(\lambda_k\|F(z_k)\|^2\bigr)^{1/2}
		\le	\sqrt{\bar{\lambda}}.
		\label{eq:zk-tk-bound}
	\end{align}
	Suppose first that $\alpha\in(0,1)$. Applying Proposition~\ref{pro3.1} with $y=t_k$ and $y'=z_k$ gives
	\begin{equation}\label{eq:operator-difference-alpha-pre}
		\|F(z_k)-F(t_k)\|\le\|z_k-t_k\|
		\left(K_0+K_1\|F(z_k)\|^\alpha
		+K_2\|z_k-t_k\|^{\frac{\alpha}{1-\alpha}}\right).
	\end{equation}
	Moreover, \eqref{eq:zk-tk-bound} implies
	\begin{equation}\label{eq:zk-tk-bound-1}
		\|z_k-t_k\|^{\frac{\alpha}{1-\alpha}}
		\le
		\bar{\lambda}^{\frac{\alpha}{2(1-\alpha)}}.
	\end{equation}
	Using \eqref{eq:operator-norm-control}, \eqref{eq:zk-tk-bound}, and \eqref{eq:zk-tk-bound-1}, we obtain
	\begin{align*}
		\sqrt{\lambda_k}\|F(z_k)-F(t_k)\|
		&\le
		\bar{\lambda}K_0
		+\sqrt{\bar{\lambda}}K_1
		\lambda_k^{\frac{1-\alpha}{2}}
		\bigl(\lambda_k\|F(z_k)\|^2\bigr)^{\alpha/2}
		+K_2\bar{\lambda}^{\frac{2-\alpha}{2(1-\alpha)}}\\
		&\le
		\bar{\lambda}K_0
		+\bar{\lambda}^{\frac{2-\alpha}{2}}K_1
		+\bar{\lambda}^{\frac{2-\alpha}{2(1-\alpha)}}K_2.
	\end{align*}
	Similarly, when $\alpha=1$, 
	\begin{align*}
		\sqrt{\lambda_k}\|F(z_k)-F(t_k)\|
		&\le
		\sqrt{\lambda_k}\|z_k-t_k\|
		\bigl(L_0+L_1\|F(z_k)\|\bigr)
		e^{L_1\|z_k-t_k\|}\\
		&\le
		\left(\bar{\lambda}L_0
		+\sqrt{\bar{\lambda}}L_1\right)
		e^{L_1\sqrt{\bar{\lambda}}}.
	\end{align*}
	Hence, in both cases,
	\begin{equation}\label{eq:operator-difference-bound}
		\lambda_k\|F(z_k)-F(t_k)\|^2
		\le M,
		\qquad k\ge1.
	\end{equation}
	Since
	\begin{equation*}
		\lambda_k^2\le\lambda_k\lambda_{k+1}
		+\lambda_k(\lambda_k-\lambda_{k+1})+,
	\end{equation*}
	Lemma~\ref{lem32}, \eqref{eq:local-difference-control}, and \eqref{eq:operator-difference-bound} imply that, for every $y\in U$,
	\begin{align*}
		\|z_{k+1}-y\|^2\le&
		\|z_k-y\|^2-\frac12\|z_k-t_k\|^2
		-2\lambda_k\langle F(t_k),t_k-y\rangle
		+M(\lambda_k-\lambda_{k+1})_+.
	\end{align*}
	Taking $y=t^*\in U^*$ and using $p$-quasi-sharpness proves \eqref{eq:descent}.
	
	\medskip
	\noindent\textbf{(ii)}
	Fix any $t^*\in U^*$. Summing \eqref{eq:descent} from $k=1$ to $n$ and dropping the nonpositive terms gives
	\begin{equation*}
		\|z_{n+1}-t^*\|^2\le\|z_1-t^*\|^2+
		M\sum_{k=1}^{n}(\lambda_k-\lambda_{k+1})+.
	\end{equation*}
	By Lemma~\ref{lem31}, the series on the right-hand side is convergent. Hence $\{z_k\}$ is bounded. Since $U$ is closed and $F$ is continuous, there exists $C>0$ such that
	\begin{equation}\label{eq:Fzk-bound}
		\|F(z_k)\|\le C,
		\qquad k\ge0.
	\end{equation}
	If $\alpha\in(0,1)$, then \eqref{eq:operator-difference-alpha-pre}, \eqref{eq:zk-tk-bound-1}, and \eqref{eq:Fzk-bound} give
	\begin{equation*}\label{eq:operator-difference-alpha}
		\|F(z_k)-F(t_k)\|\le\|z_k-t_k\|
		\left(K_0+K_1C^\alpha
		+K_2\bar{\lambda}^{\frac{\alpha}{2(1-\alpha)}}\right),
		\qquad k\ge1.
	\end{equation*}
	Similarly, if $\alpha=1$,
	\begin{equation*}\label{eq:operator-difference-one}
		\|F(z_k)-F(t_k)\|\le\|z_k-t_k\|
		(L_0+L_1C)e^{L_1\sqrt{\bar{\lambda}}},
		\qquad k\ge1.
	\end{equation*}
	Therefore, whenever $F(z_k)\ne F(t_k)$,
	\begin{equation*}
		\frac{\|z_k-t_k\|^2}
		{2\lambda_k\|F(z_k)-F(t_k)\|^2}
		\ge
		\begin{cases}
			\displaystyle
			\frac{1}{
				2\bar{\lambda}
				\left(
				K_0+K_1C^\alpha
				+K_2\bar{\lambda}^{\frac{\alpha}{2(1-\alpha)}}
				\right)^2},
			& \alpha\in(0,1),\\[4mm]
			\displaystyle
			\frac{1}{
				2\bar{\lambda}(L_0+L_1C)^2
				e^{2L_1\sqrt{\bar{\lambda}}}},
			& \alpha=1.
		\end{cases}
	\end{equation*}
	If $F(z_k)=F(t_k)$, the corresponding candidate is interpreted as $+\infty$.
	By the definition of $\underline{\lambda}$, 
	$\lambda_0\ge\underline{\lambda}$ and
	$\lambda_1\ge\underline{\lambda}$. Suppose that
	$\lambda_k\ge\underline{\lambda}$ for some $k\ge1$. The first
	term in \eqref{eq:pf-aeg-stepsize} satisfies
	$(1+\eta_k)\lambda_k\ge\underline{\lambda}$. The preceding estimates show that the second candidate in \eqref{eq:pf-aeg-stepsize} is also no smaller than $\underline{\lambda}$.  Finally, by \eqref{eq:Fzk-bound}, the third candidate satisfies
	$1/\|F(z_{k+1})\|^2\ge1/C^2\ge\underline{\lambda}$, with the
	convention $1/0=+\infty$. Hence
	$\lambda_{k+1}\ge\underline{\lambda}$. By induction,
	\eqref{eq:lambda-lower} holds for all $k\ge0$.

	\medskip
	\noindent\textbf{(iii)}
	It follows from \eqref{eq:local-difference-control} and
	Lemma~\ref{lem32} that, for every $y\in U$,
	\begin{equation}\label{6.9}
			\|z_{k+1}-y\|^2\le\|z_k-y\|^2
			-\left(1-\frac{\lambda_k}{2\lambda_{k+1}}\right)\|z_k-t_k\|^2
			-2\lambda_k\langle F(t_k),t_k-y\rangle.
	\end{equation}
By Lemma~\ref{lem31} and \eqref{eq:lambda-lower}, $\{\lambda_k\}$ converges to a positive limit. Hence $\frac{\lambda_k}{\lambda_{k+1}}\to1$.
	Therefore, there exists $k_0\ge1$ such that
	\begin{equation*}\label{eq:ratio-control}
		1-\frac{\lambda_k}{2\lambda_{k+1}}
		\ge\frac14,
		\qquad k\ge k_0.
	\end{equation*}
	Fix an arbitrary $t^*\in U^*$. Taking $y=t^*\in U^*$ in \eqref{6.9} and applying $p$-quasi-sharpness and \eqref{eq:lambda-lower} give
	\begin{equation}\label{6.30}
			\|z_{k+1}-t^*\|^2\le\|z_k-t^*\|^2
			-\frac14\|z_k-t_k\|^2
			-2\mu\underline{\lambda}
			\operatorname{dist}^p(t_k,U^*),
			\qquad k\ge k_0.
	\end{equation}
	Applying Lemma~\ref{lem23} to the tail sequence in \eqref{6.30} shows that $\{\|z_k-t^*\|^2\}$ converges and
	\begin{equation*}\label{eq:summability-tail}
		\sum_{k=k_0}^{\infty}\|z_k-t_k\|^2<\infty,
		\qquad
		\sum_{k=k_0}^{\infty}\operatorname{dist}^p(t_k,U^*)<\infty.
	\end{equation*}
	Since only finitely many indices precede $k_0$, we get
	\begin{equation*}\label{eq:summability-global}
		\sum_{k=0}^{\infty}\|z_k-t_k\|^2<\infty,
		\qquad
		\sum_{k=0}^{\infty}\operatorname{dist}^p(t_k,U^*)<\infty.
	\end{equation*}
	Consequently,
	\begin{equation}\label{6.31}
		\|z_k-t_k\|\to0,
		\qquad
		\operatorname{dist}(t_k,U^*)\to0.
	\end{equation}
	
	We next show that every accumulation point of
	$\{z_k\}$ belongs to $U^*$.
	By part $\mathbf{(ii)}$, the sequence $\{z_k\}$ is bounded.
	Hence it has an accumulation point in $\mathbb{R}^m$.
	Let $\{z_{k_i}\}$ be a convergent subsequence such that $z_{k_i}\to\bar t$. 
	Since $z_{k_i}\in U$ and $U$ is closed, we have
	$\bar t\in U$.
	Moreover, \eqref{6.31} gives $\|t_{k_i}-z_{k_i}\|\to0$, 
	and therefore $t_{k_i}\to\bar t$. 
	Since the distance function is continuous,
	\eqref{6.31} yields
	\begin{equation*}
		\operatorname{dist}(\bar t,U^*)=
		\lim_{i\to\infty}\operatorname{dist}(t_{k_i},U^*)=0.
	\end{equation*}
	Since $U^*$ is closed, $\bar t\in U^*$. Thus every accumulation point of $\{z_k\}$ belongs to $U^*$.
	
	Finally, \eqref{6.30} implies that the tail sequence $\{z_k\}_{k\ge k_0}$ 
	is Fej\'er monotone with respect to $U^*$.
	Applying Lemma~\ref{lem:Fejer} to this tail sequence, we conclude
	that there exists $\bar t\in U^*$ such that $z_k\to\bar t$. 
	Together with $\|z_k-t_k\|\to0$, this further gives $t_k\to\bar t$. 
	Hence both $\{z_k\}$ and $\{t_k\}$ converge to the same
	solution $\bar t\in U^*$.
\end{proof}

\begin{theorem}\label{th33}
	Suppose that Assumptions~\ref{ass1}--\ref{ass3} hold, and let
	$\{t_k\}$ and $\{z_k\}$ be generated by Algorithm~\ref{alg3}. Then the
	following statements hold.
	
	\begin{enumerate}
		\item[$\mathbf{(i)}$]
		If $0<p\le2$, then there exists $c>0$ such that
		\begin{equation}\label{eq:global-linear-rate}
			\operatorname{dist}^2(z_k,U^*)\le c\rho^k,
			\qquad k\ge0,
		\end{equation}
		where
		\begin{equation}\label{eq:rho}
			\rho:=1-\min\left\{\frac18,\,
			\frac{\mu\underline{\lambda}}{B^{2-p}}\right\}\in(0,1),
		\end{equation}
		and $B>0$ is a constant satisfying
		$\operatorname{dist}(t_k,U^*)\le B$ for all $k\ge0$.
		
		\item[$\mathbf{(ii)}$]
		If $p>2$, then there exists $c_p>0$ such that
		\begin{equation}\label{eq:global-sublinear-rate}
			\operatorname{dist}^2(z_k,U^*)
			\le\frac{c_p}{(k+1)^{2/(p-2)}},
			\qquad k\ge0.
		\end{equation}
	\end{enumerate}
\end{theorem}

\begin{proof}
	By Theorem~\ref{th61}(iii), there exists $\bar t\in U^*$ such that $t_k\to\bar t$. Hence $\operatorname{dist}(t_k,U^*)\to0$.
	Therefore, there exists $B>0$ such that
	\begin{equation}\label{eq:B-bound}
		\operatorname{dist}(t_k,U^*)\le B,
		\qquad k\ge0.
	\end{equation}
	Let $k_0\ge1$ be such that \eqref{6.30} holds for all $k\ge k_0$. For each such $k$, let $z_k^*$ be the projection of $z_k$ onto	$U^*$, i.e., 
	\begin{equation*}
		\|z_k-z_k^*\|
		=
		\operatorname{dist}(z_k,U^*).
	\end{equation*}
	Taking $t^*=z_k^*$ in \eqref{6.30} gives
	\begin{equation}\label{eq:rate-basic}
		\operatorname{dist}^2(z_{k+1},U^*)
		\le
		\operatorname{dist}^2(z_k,U^*)
		-\frac14\|z_k-t_k\|^2
		-2\mu\underline{\lambda}
		\operatorname{dist}^p(t_k,U^*),
		\qquad k\ge k_0.
	\end{equation}
	Moreover, the triangle inequality yields
	\begin{equation}\label{eq:distance-split}
		\operatorname{dist}(z_k,U^*)
		\le\|z_k-t_k\|+\operatorname{dist}(t_k,U^*).
	\end{equation}
	
	\medskip
	\noindent\textbf{(i)}
	Suppose $0<p\le2$. Squaring \eqref{eq:distance-split} and using
	\eqref{eq:B-bound}, we obtain
	\begin{align*}
		\operatorname{dist}^2(z_k,U^*)
		&\le2\|z_k-t_k\|^2
		+2\operatorname{dist}^2(t_k,U^*)
		\le2\|z_k-t_k\|^2+2B^{2-p}\operatorname{dist}^p(t_k,U^*).
	\end{align*}
	By the definition of $\rho$ in \eqref{eq:rho},
	\begin{equation*}
		2(1-\rho)\le\frac14,
		\qquad
		2(1-\rho)B^{2-p}\le2\mu\underline{\lambda}.
	\end{equation*}
	Hence
	\begin{equation*}
		(1-\rho)\operatorname{dist}^2(z_k,U^*)
		\le\frac14\|z_k-t_k\|^2
		+2\mu\underline{\lambda}
		\operatorname{dist}^p(t_k,U^*).
	\end{equation*}
	Combining this inequality with \eqref{eq:rate-basic} yields
	\begin{equation*}
		\operatorname{dist}^2(z_{k+1},U^*)
		\le\rho\,\operatorname{dist}^2(z_k,U^*),
		\qquad k\ge k_0,
	\end{equation*}
	which implies
	\begin{equation*}\label{eq:linear-tail-rate}
		\operatorname{dist}^2(z_k,U^*)
		\le\rho^{\,k-k_0}
		\operatorname{dist}^2(z_{k_0},U^*),
		\qquad k\ge k_0.
	\end{equation*}
	Then \eqref{eq:global-linear-rate} holds for every $k\ge0$ by setting 
	\begin{equation*}
		c:=\max_{0\le j\le k_0}\rho^{-j}\operatorname{dist}^2(z_j,U^*).
	\end{equation*}

	\medskip
	\noindent\textbf{(ii)}
	Suppose $p>2$. By \eqref{eq:distance-split}, the fact that ${\left(a + b\right)}^{p} \leq  {2}^{p - 1}\left(a^{p} + b^{p}\right)$, and
	$\|z_k-t_k\|\le\sqrt{\bar{\lambda}}$, we have
	\begin{align*}
		\operatorname{dist}^p(z_k,U^*)
		&\le\left(\|z_k-t_k\|+\operatorname{dist}(t_k,U^*)\right)^p\\
		&\le2^{p-1}\left(\|z_k-t_k\|^p+\operatorname{dist}^p(t_k,U^*)
		\right)\\
		&\le2^{p-1}\left(\bar{\lambda}^{\frac{p-2}{2}}
		\|z_k-t_k\|^2+\operatorname{dist}^p(t_k,U^*)\right).
	\end{align*}
	Thus
	\begin{align*}
		&2^{2-p}
		\min\left\{
			\mu\underline{\lambda},
			\frac{1}{8\bar{\lambda}^{(p-2)/2}}
			\right\}
		\operatorname{dist}^p(z_k,U^*)
		\le\frac14\|z_k-t_k\|^2
		+2\mu\underline{\lambda}
		\operatorname{dist}^p(t_k,U^*).
	\end{align*}
	Combining this estimate with \eqref{eq:rate-basic} gives
	\begin{equation}\label{eq:sublinear-tail-recursion}
		\operatorname{dist}^2(z_{k+1},U^*)
		\le\operatorname{dist}^2(z_k,U^*)
		-2^{2-p}\min\left\{
			\mu\underline{\lambda},,
			\frac{1}{8\bar{\lambda}^{(p-2)/2}}
			\right\}
		\operatorname{dist}^p(z_k,U^*),
		\qquad k\ge k_0.
	\end{equation}
	Applying Lemma~\ref{lem25} to \eqref{eq:sublinear-tail-recursion} with $q=(p-2)/2$ yields
	\begin{equation}\label{eq:sublinear-tail-rate}
		\operatorname{dist}^2(z_k,U^*)
		\le\frac{\operatorname{dist}^2(z_{k_0},U^*)}{
			\left[1+a_p(k-k_0)
			\operatorname{dist}^{p-2}(z_{k_0},U^*)
			\right]^{2/(p-2)}
		},
		\qquad k\ge k_0,
	\end{equation}
	where
	\begin{equation*}
		a_p=2^{1-p}(p-2)
		\min\left\{
			\mu\underline{\lambda},
			\frac{1}{8\bar{\lambda}^{(p-2)/2}}
			\right\}.
	\end{equation*}
	If $\operatorname{dist}(z_{k_0},U^*)=0$, then
	\eqref{eq:sublinear-tail-recursion} implies
	$\operatorname{dist}(z_k,U^*)=0$ for all $k\ge k_0$, and the
	result follows immediately. 
	Otherwise, \eqref{eq:sublinear-tail-rate} implies, for $k\ge k_0+1$,
	\begin{equation*}
		\operatorname{dist}^2(z_k,U^*)
		\le\frac{1}{[a_p(k-k_0)]^{2/(p-2)}}.
	\end{equation*}
	For $k\ge k_0+1$, $k-k_0\ge\frac{k+1}{k_0+2}$,
	and hence
	\begin{equation*}
		\operatorname{dist}^2(z_k,U^*)\le
		\left(\frac{k_0+2}{a_p}\right)^{2/(p-2)}
		\frac{1}{(k+1)^{2/(p-2)}}.
	\end{equation*}
	Thus \eqref{eq:global-sublinear-rate} holds for every $k\ge0$ with
	\begin{equation*}
		c_p=\max\left\{
			\max_{0\le j\le k_0}
			(j+1)^{2/(p-2)}
			\operatorname{dist}^2(z_j,U^*),
			\left(\frac{k_0+2}{a_p}\right)^{2/(p-2)}
			\right\}.
	\end{equation*}
\end{proof}

\section{Variance-Reduced Adaptive Extragradient Method}
\label{sec:aeg-vr}

We next consider VIs with the finite-sum operator of the form
\eqref{finite-sum}. The following assumption will be used in the convergence analysis.

\begin{assumption}\label{ass4}
	For every $i\in[N]:=\{1,\ldots,N\}$, the component operator
	$F_i:U\to\mathbb{R}^m$ is continuous and satisfies the
	$\alpha$-symmetric generalized-smoothness condition for a common
	$\alpha\in(0,1]$. In particular, there exist
	$L_{0,i},L_{1,i}\ge0$ such that
	\begin{equation}\label{eq:component-generalized-smoothness}
		\|F_i(u)-F_i(v)\|\le
		\left(L_{0,i}+L_{1,i}\max_{\theta\in[0,1]}
		\|F_i(\theta u+(1-\theta)v)\|^\alpha\right)\|u-v\|,
		\qquad u,v\in U.
	\end{equation}
\end{assumption}

For $\alpha\in(0,1)$, let $K_{0,i},K_{1,i},K_{2,i}$ denote the
constants in Proposition~\ref{pro3.1} associated with $F_i$, and define
\begin{equation*}
	K_j:=\max_{i\in[N]}K_{j,i},
	\qquad j=0,1,2.
\end{equation*}
For $\alpha=1$, define
\begin{equation*}
	L_0:=\max_{i\in[N]}L_{0,i},
	\qquad
	L_1:=\max_{i\in[N]}L_{1,i}.
\end{equation*}

\subsection{Algorithmic Framework}

AEG-VR combines the adaptive stepsize principle of PF-AEG with an
SVRG-type estimator and a randomized snapshot update. The snapshot $s^k$ stores the full operator $F(s^k)$. A safeguard forces a refresh when the new iterate moves too far from the current snapshot; otherwise the snapshot is refreshed with probability $P_1$.

\begin{algorithm}[ht]
	\caption{Adaptive Extragradient Method with Variance Reduction (AEG-VR)}
	\label{alg4}
	
	\textbf{Initialization:}
	Choose $t^0=s^0\in U$, $\lambda_0>0$,
	$\beta\in(0,\frac12)$, $P_1\in(0,1)$, and $w>0$.
	Set $k=0$.
	
	\textit{Step 1:}
	Compute
	\begin{equation*}
		\bar t^k=\beta t^k+(1-\beta)s^k,
	\end{equation*}
	\begin{equation*}
		t^{k+\frac12}=P_U\!\left(\bar t^k-\lambda_kF(s^k)\right).
	\end{equation*}
	
	\textit{Step 2:}
	Sample $\xi_k$ uniformly from $[N]$, independently of the past, and
	compute
	\begin{equation}\label{eq:svrg-estimator}
		\widehat F_k=
		F(s^k)+F_{\xi_k}\!\left(t^{k+\frac12}\right)-F_{\xi_k}(s^k),
	\end{equation}
	\begin{equation}\label{eq:aeg-vr-correction}
		t^{k+1}=P_U\left(\bar t^k-\lambda_k\widehat F_k\right).
	\end{equation}
	
	\textit{Step 3:}
	If   $\left\|t^{k+1} -s^{k}\right\| > w$,  then  set    $s_{k+1}=t_{k+1}$.
	Otherwise, set
	\begin{equation}\label{eq:snapshot-update}
	s^{k+1}=
	\begin{cases}
	t^{k+1}, & \text{with probability } P_{1}, \\
	s^k, & \text{with probability } 1 - P_{1}.                                                                                                                                                                                                                                                                                                                                                                                                                                                                                                                                                                                                                                                                                                                                                                                                                                                                                                                                                                                                                                                                                                                                                                                                                                                                                                                                                                                                                                                                                                                                                                                                                                                                                                                                                                                                                                                                                                                                                                                                                                                                                                                                                                                                                                                                                                                                                                                                                                                                                                                                                                                                                                                                                                                                                                                                                                                                                                                                                                                                                                                                                                                                                                                                                                                                                                                                                                                                                                                                                                                                                                                                                                                                                                                                                                                                                                                                                                                                                                                                                                                                                                                                                                                                                                                                                                                                                                                                                                                                                                                                                                                                                                                                                                                                                                                                                                                                                                                                                                                                                                                                                                                                                                                                                                                                                                                                                                                                                                                                                                                                                                                                                                                                                                                                                                                                                                                                                                                                                                                                                                                                                                                                                                                                                                                                                                                                                                                                                                                                                                                                                                                                                                                                                                                                                                                                                                                                                                                                                                                                                                                                                                                                                                                                                                                                                                                                                                                                                                                                                                                                                                                                                                                                                                                                                                                                                                                                                                                                                                                                                                                                                                                                                                                                                                                                                                                                                                                                                                                                                                                                                                                                                                                                                                                                                                                                                                                                                                                                                                                                                                                                                                                                                                                                                                                                                                                                                                                                                                                               
\end{cases}
	\end{equation}
	
	\textit{Step 4:}
	Set $\eta_k=\frac{1}{(k+1)^{2}}$ and update
	\begin{equation}\label{eq:aeg-vr-stepsize}
		\lambda_{k+1}=
		\min\left\{(1+\eta_k)\lambda_k,\,
		\frac{\beta\|t^{k+\frac12}-s^k\|}
		{\|F_{\xi_k}(t^{k+\frac12})-F_{\xi_k}(s^k)\|},\,
		\frac{1}{\max_{i\in[N]}\|F_i(s^{k+1})\|^2}
		\right\}.
	\end{equation}

	\textit{Step 5:}
	Set $k\gets k+1$ and return to Step~1.
\end{algorithm}

A candidate stepsize with a zero denominator is interpreted as $+\infty$. The third candidate in \eqref{eq:aeg-vr-stepsize} indicates
\begin{equation}\label{eq:stochastic-operator-control}
	\lambda_k\max_{i\in[N]}\|F_i(s^k)\|^2\le1,
	\qquad k\ge1.
\end{equation}
The snapshot rule also gives, by induction,
\begin{equation}\label{eq:snapshot-gap}
	\|t^k-s^k\|\le w,
	\qquad k\ge0.
\end{equation}
Consequently, as in the analysis of PF-AEG, the component-norm control is
available after the first iteration, whereas the snapshot-gap control holds
from initialization. 

\subsection{Convergence Analysis}

The filtration must include both sources of randomness. Define
\begin{equation*}
	\mathcal F_k:=\sigma(\xi_0,\zeta_0,\ldots,\xi_{k-1},\zeta_{k-1}),
	\qquad
	\mathcal G_k:=\mathcal F_k\vee\sigma(\xi_k).
\end{equation*}
Then $t^k,s^k,\lambda_k,\bar t^k,t^{k+\frac12}$ are
$\mathcal F_k$-measurable; $t^{k+1}$ and
$\{\|t^{k+1}-s^k\|>w\}$ are $\mathcal G_k$-measurable; and
$s^{k+1},\lambda_{k+1}$ are $\mathcal F_{k+1}$-measurable. We write
\begin{equation*}
	\mathbb E_k[\cdot]:=
	\mathbb E[\cdot\mid\mathcal F_k],
	\qquad
	\mathbb E_k^+[\cdot]:=
	\mathbb E[\cdot\mid\mathcal G_k].
\end{equation*}
The estimator \eqref{eq:svrg-estimator} is conditionally unbiased:
\begin{equation}\label{eq:svrg-unbiased}
	\mathbb E_k[\widehat F_k]=
	F\!\left(t^{k+\frac12}\right).
\end{equation}

\begin{lemma}\label{lem41}
	Let $\{\lambda_k\}$ be generated by Algorithm~\ref{alg4}. Then, almost
	surely,
	\begin{equation}\label{eq:vr-lambda-upper}
		0<\lambda_k\le
		\bar\lambda:=
		\lambda_0\exp\!\left(\sum_{j=0}^{\infty}\eta_j\right),
		\qquad k\ge0,
	\end{equation}
	the limit $\lambda_k\to\lambda_\infty\ge0$ exists, and
	\begin{equation}\label{eq:vr-lambda-variation}
		\sum_{k=0}^{\infty}|\lambda_{k+1}-\lambda_k|<\infty.
	\end{equation}
	Moreover,
	\begin{equation}\label{eq:vr-lambda-variation-expectation}
		\sum_{k=0}^{\infty}\mathbb E(\lambda_{k+1}-\lambda_k)_+<\infty,
		\qquad
		\sum_{k=0}^{\infty}\mathbb E(\lambda_k-\lambda_{k+1})_+<\infty.
	\end{equation}
\end{lemma}

\begin{proof}
	The first candidate in \eqref{eq:aeg-vr-stepsize} gives
	$\lambda_{k+1}\le(1+\eta_k)\lambda_k$ pathwise. Hence the proof of
	\eqref{eq:vr-lambda-upper} is identical to that of Lemma~\ref{lem31}.
	Also,
	\begin{equation*}
		(\lambda_{k+1}-\lambda_k)_+\le\bar\lambda\eta_k.
	\end{equation*}
	Since $\sum_k\eta_k<\infty$, the positive variations $(\lambda_{k+1}-\lambda_k)_+$ are summable
	pathwise. The identity
	\begin{equation*}
		\lambda_{N+1}-\lambda_0=
		\sum_{k=0}^{N}(\lambda_{k+1}-\lambda_k)_+
		-\sum_{k=0}^{N}(\lambda_{k+1}-\lambda_k)_-
	\end{equation*}
	and $\lambda_{N+1}\ge0$ imply summability of the negative variations.
	This proves \eqref{eq:vr-lambda-variation}, and therefore
	$\{\lambda_k\}$ is Cauchy and convergent. The expectation bounds in
	\eqref{eq:vr-lambda-variation-expectation} follow from the same deterministic
	majorants and Tonelli's theorem.
\end{proof}

\begin{lemma}\label{lem42}
	Suppose Assumptions~\ref{ass1}--\ref{ass4} hold. For every $t^*\in U^*$
	and every $k\ge1$,
	\begin{align}\label{eq:vr-one-step}
			\mathbb E_k\|t^{k+1}-t^*\|^2
			\le&
			\beta\|t^k-t^*\|^2+(1-\beta)\|s^k-t^*\|^2
			-2\mu\lambda_k\operatorname{dist}^p(t^{k+\frac12},U^*)\notag\\
			&-(1-\beta)\mathbb E_k\|t^{k+1}-t^{k+\frac12}\|^2
			-(1-2\beta)\|s^k-t^{k+\frac12}\|^2
			-\beta\|t^k-t^{k+\frac12}\|^2\notag\\
			&+\bar M\,\mathbb E_k(\lambda_k-\lambda_{k+1})_+.
	\end{align}
	where $R=\beta w+\sqrt{\bar\lambda}$ and
	\begin{equation}\label{eq:Mbar}
		\bar M=
		\begin{cases}
			\displaystyle
			\frac{2R^2}{\beta}
			\left[\sqrt{\bar\lambda}
			\left(K_0+K_2R^{\frac{\alpha}{1-\alpha}}\right)
			+\bar\lambda^{\frac{1-\alpha}{2}}K_1
			\right]^2,
			& \alpha\in(0,1),\\[4mm]
			\displaystyle
			\frac{2R^2}{\beta}
			\left[(\sqrt{\bar\lambda}L_0+L_1)e^{L_1R}\right]^2,
			& \alpha=1.
		\end{cases}
	\end{equation}
\end{lemma}

\begin{proof}
	Fix any $t^*\in U^*$ and $k\ge1$. Since
	$t^{k+1}=P_U\left(\bar t^k-\lambda_k\widehat F_k\right)$, 
	the inequality \eqref{2.2} gives
	\begin{align*}
		\|t^{k+1}-t^*\|^2
		\le{}&\|\bar t^k-t^*\|^2
		-\|t^{k+1}-\bar t^k\|^2+2\lambda_k
		\langle \widehat F_k,t^*-t^{k+1}\rangle.
	\end{align*}
	Splitting $t^*-t^{k+1}=t^*-t^{k+\frac12}+t^{k+\frac12}-t^{k+1}$,	we obtain
	\begin{align}
		\|t^{k+1}-t^*\|^2
		\le{}&\|\bar t^k-t^*\|^2
		-\|t^{k+1}-\bar t^k\|^2-2\lambda_k
		\langle t^{k+\frac12}-t^*,\widehat F_k\rangle
		+2\lambda_k\langle\widehat F_k,t^{k+\frac12}-t^{k+1}\rangle.
		\label{eq:vr-projection-split}
	\end{align}
	Applying \eqref{2.1} with	$u=t^{k+1}\in U$ and $v=\bar t^k-\lambda_kF(s^k)$ yields
	\begin{equation*}
		\left\langle
		\bar t^k-\lambda_kF(s^k)-t^{k+\frac12},
		t^{k+1}-t^{k+\frac12}
		\right\rangle
		\le0.
	\end{equation*}
	Equivalently,
	\begin{equation*}\label{eq:vr-prediction-optimality}
		\lambda_k\left\langle F(s^k),t^{k+\frac12}-t^{k+1}\right\rangle
		\le\left\langle\bar t^k-t^{k+\frac12},t^{k+\frac12}-t^{k+1}\right\rangle.
	\end{equation*}
	Using
	$\widehat F_k=F(s^k)+F_{\xi_k}(t^{k+\frac12})-F_{\xi_k}(s^k)$, it follows that
	\begin{align*}
		\lambda_k\langle\widehat F_k,t^{k+\frac12}-t^{k+1}\rangle
		\le&\left\langle\bar t^k-t^{k+\frac12},
		t^{k+\frac12}-t^{k+1}\right\rangle
		+\lambda_k\left\langle
		F_{\xi_k}(t^{k+\frac12})-F_{\xi_k}(s^k),
		t^{k+\frac12}-t^{k+1}\right\rangle.
	\end{align*}
	Substituting this estimate into
	\eqref{eq:vr-projection-split} and using the identity
	\begin{align*}
		&-\|t^{k+1}-\bar t^k\|^2+
		2\left\langle\bar t^k-t^{k+\frac12},t^{k+\frac12}-t^{k+1}\right\rangle
		=-\|t^{k+1}-t^{k+\frac12}\|^2
		-\|\bar t^k-t^{k+\frac12}\|^2,
	\end{align*}
	we obtain
	\begin{align}
		\|t^{k+1}-t^*\|^2
		\le{}&\|\bar t^k-t^*\|^2
		-\|t^{k+1}-t^{k+\frac12}\|^2
		-\|\bar t^k-t^{k+\frac12}\|^2
		\notag\\
		&-2\lambda_k\left\langle
		t^{k+\frac12}-t^*,\widehat F_k\right\rangle
		+2\lambda_k\left\langle
		F_{\xi_k}(t^{k+\frac12})-F_{\xi_k}(s^k),
		t^{k+\frac12}-t^{k+1}\right\rangle.
		\label{eq:vr-before-young}
	\end{align}
	By Young's inequality,
	\begin{equation*}
		2\lambda_k\left\langle
		F_{\xi_k}(t^{k+\frac12})-F_{\xi_k}(s^k),
		t^{k+\frac12}-t^{k+1}\right\rangle
		\le\beta\|t^{k+1}-t^{k+\frac12}\|^2
		+\frac{\lambda_k^2}{\beta}\|F_{\xi_k}(t^{k+\frac12})-F_{\xi_k}(s^k)\|^2.
	\end{equation*}
	Therefore,
	\begin{align}
		\|t^{k+1}-t^*\|^2
		\le{}&\|\bar t^k-t^*\|^2
		-(1-\beta)\|t^{k+1}-t^{k+\frac12}\|^2
		-\|\bar t^k-t^{k+\frac12}\|^2
		\notag\\
		&-2\lambda_k\left\langle
		t^{k+\frac12}-t^*,\widehat F_k\right\rangle
		+\frac{\lambda_k^2}{\beta}\|F_{\xi_k}(t^{k+\frac12})-F_{\xi_k}(s^k)\|^2.
		\label{eq:vr-after-young}
	\end{align}
	Since $\bar t^k=\beta t^k+(1-\beta)s^k$, the standard convex-combination identity gives, for every $u\in\mathbb R^m$,
	\begin{equation*}
		\|\bar t^k-u\|^2=\beta\|t^k-u\|^2
		+(1-\beta)\|s^k-u\|^2
		-\beta(1-\beta)\|t^k-s^k\|^2.
	\end{equation*}
	Applying this identity with $u=t^*$ and
	$u=t^{k+\frac12}$, and subtracting the resulting equalities, yields
	\begin{align*}
		\|\bar t^k-t^*\|^2
		-\|\bar t^k-t^{k+\frac12}\|^2
		={}&\beta\|t^k-t^*\|^2
		+(1-\beta)\|s^k-t^*\|^2\\
		&-\beta\|t^k-t^{k+\frac12}\|^2
		-(1-\beta)\|s^k-t^{k+\frac12}\|^2.
	\end{align*}
	Substituting this equality into
	\eqref{eq:vr-after-young} gives
	\begin{align}
		\|t^{k+1}-t^*\|^2
		\le{}&\beta\|t^k-t^*\|^2
		+(1-\beta)\|s^k-t^*\|^2
		-2\lambda_k\left\langle
		t^{k+\frac12}-t^*,\widehat F_k\right\rangle
		\notag\\
		&-(1-\beta)
		\|t^{k+1}-t^{k+\frac12}\|^2
		-(1-\beta)\|s^k-t^{k+\frac12}\|^2
		-\beta\|t^k-t^{k+\frac12}\|^2
		\notag\\
		&+\frac{\lambda_k^2}{\beta}\|F_{\xi_k}(t^{k+\frac12})-F_{\xi_k}(s^k)\|^2.
		\label{eq:vr-prelocal}
	\end{align}

	We next control the last term. By the second candidate in
	\eqref{eq:aeg-vr-stepsize},
	\begin{equation*}
		\frac{\lambda_{k+1}^2}{\beta}
		\|F_{\xi_k}(t^{k+\frac12})-F_{\xi_k}(s^k)\|^2
		\le\beta\|t^{k+\frac12}-s^k\|^2.
	\end{equation*}
	Using $\lambda_k^2\le\lambda_{k+1}^2+2\lambda_k(\lambda_k-\lambda_{k+1})_+$, we have
	\begin{align}
		\frac{\lambda_k^2}{\beta}
		\|F_{\xi_k}(t^{k+\frac12})-F_{\xi_k}(s^k)\|^2
		\le{}&\beta\|t^{k+\frac12}-s^k\|^2
		+\frac{2\lambda_k}{\beta}
		\|F_{\xi_k}(t^{k+\frac12})-F_{\xi_k}(s^k)\|^2
		(\lambda_k-\lambda_{k+1})_+.
		\label{eq:vr-error-split}
	\end{align}
	Since $s^k\in U$,	nonexpansiveness of $P_U$, \eqref{eq:snapshot-gap}, and
	\eqref{eq:stochastic-operator-control} yield
	\begin{align}
		\|t^{k+\frac12}-s^k\|
		&=\left\|P_U\!\left(\bar t^k-\lambda_kF(s^k)\right)-P_U(s^k)\right\|
		\notag\\
		&\le\beta\|t^k-s^k\|+\lambda_k\|F(s^k)\|
		\notag\\
		&\le\beta w+\lambda_k\max_{i\in[N]}\|F_i(s^k)\|
		\notag\\
		&\le\beta w+\sqrt{\bar\lambda}=R.
		\label{eq:vr-prediction-gap}
	\end{align}
	If $\alpha\in(0,1)$, by
	Assumption~\ref{ass4}, Proposition~\ref{pro3.1}, and
	\eqref{eq:vr-prediction-gap},
	\begin{align*}
		&\sqrt{\frac{2\lambda_k}{\beta}}\,
		\|F_{\xi_k}(t^{k+\frac12})-F_{\xi_k}(s^k)\|\\
		&\quad\le\sqrt{\frac{2}{\beta}}\,
		\|t^{k+\frac12}-s^k\|
		\Bigl[\sqrt{\lambda_k}
		\left(K_0+K_2\|t^{k+\frac12}-s^k\|^{\frac{\alpha}{1-\alpha}}\right)
		+K_1\sqrt{\lambda_k}\,\|F_{\xi_k}(s^k)\|^\alpha\Bigr]\\
		&\quad\le\sqrt{\frac{2}{\beta}}\,R
		\left[\sqrt{\bar\lambda}
		\left(K_0+K_2R^{\frac{\alpha}{1-\alpha}}\right)
		+\bar\lambda^{\frac{1-\alpha}{2}}K_1\right].
	\end{align*}
	If $\alpha=1$, Proposition~\ref{pro3.1} similarly gives
	\begin{align*}
		&\sqrt{\frac{2\lambda_k}{\beta}}\,
		\|F_{\xi_k}(t^{k+\frac12})-F_{\xi_k}(s^k)\|
		\le\sqrt{\frac{2}{\beta}}\,R
		\left(\sqrt{\bar\lambda}L_0+L_1\right)e^{L_1R},
	\end{align*}
	where we used $\sqrt{\lambda_k}\|F_{\xi_k}(s^k)\|
	\le1$. Thus, in both cases, by the definition of $\bar M$,
	\begin{equation}\label{eq:vr-component-bound}
		\frac{2\lambda_k}{\beta}
		\|F_{\xi_k}(t^{k+\frac12})-F_{\xi_k}(s^k)\|^2
		\le\bar M.
	\end{equation}
	Combining \eqref{eq:vr-error-split} and
	\eqref{eq:vr-component-bound} gives
	\begin{align*}
		\frac{\lambda_k^2}{\beta}
		\|F_{\xi_k}(t^{k+\frac12})-F_{\xi_k}(s^k)\|^2
		\le&\beta\|t^{k+\frac12}-s^k\|^2
		+\bar M(\lambda_k-\lambda_{k+1})_+.
	\end{align*}
	Substituting this estimate into \eqref{eq:vr-prelocal} yields
	\begin{align*}
		\|t^{k+1}-t^*\|^2
		\le{}&\beta\|t^k-t^*\|^2
		+(1-\beta)\|s^k-t^*\|^2-2\lambda_k
		\langle t^{k+\frac12}-t^*,\widehat F_k\rangle\\
		&-(1-\beta)\|t^{k+1}-t^{k+\frac12}\|^2
		-(1-2\beta)\|s^k-t^{k+\frac12}\|^2
		-\beta\|t^k-t^{k+\frac12}\|^2\\
		&+\bar M(\lambda_k-\lambda_{k+1})_+.
	\end{align*}
	
	Finally, taking conditional expectation with respect to
	$\mathcal F_k$, using \eqref{eq:svrg-unbiased}, gives
	\begin{equation*}
		-\mathbb E_k
		\langle t^{k+\frac12}-t^*,\widehat F_k\rangle
		=-\langle t^{k+\frac12}-t^*,F(t^{k+\frac12})\rangle
		\le-\mu\operatorname{dist}^p(t^{k+\frac12},U^*),
	\end{equation*}
	where the last inequality is due to $p$-quasi-sharpness.
	Substituting this inequality into the preceding estimate proves
	\eqref{eq:vr-one-step}.
\end{proof}

\begin{lemma}\label{lem:finite-threshold}
	Assume
	\begin{equation}\label{eq:w-condition}
		w>\frac{1+\beta}{1-\beta-\beta^2}\sqrt{\bar\lambda}.
	\end{equation}
	Then the safeguard event
	\begin{equation*}
		A_k:=\{\|t^{k+1}-s^k\|>w\}
	\end{equation*}
	occurs only finitely many times almost surely.
\end{lemma}

\begin{proof}
	For every $k\ge1$, by \eqref{2.3},  \eqref{eq:aeg-vr-correction}  and \eqref{eq:aeg-vr-stepsize},  we have 
	\begin{align}
		\|t^{k+1}-s^{k}\|
		\leq & \|t^{k+1}-t^{k+\frac{1}{2}}\|
		+\|t^{k+\frac{1}{2}}-s^{k}\|\notag \\
		\leq & \|t^{k+\frac{1}{2}}-s^{k}\|
		+\lambda_k\|F_{\xi_k}(t^{k+\frac{1}{2}})-F_{\xi_k}(s^k)\|\notag \\
		\leq &(1+\beta)\|t^{k+\frac{1}{2}}-s^{k}\|
		+(\lambda_k-\lambda_{k+1})_+
		\|F_{\xi_k}(t^{k+\frac{1}{2}})-F_{\xi_k}(s^k)\|,
		\label{f1}
	\end{align}
	where the last inequality follows from the second inequality by adding and subtracting the term $\lambda_{k+1}\|F_{\xi_k}(t^{k+\frac{1}{2}})-F_{\xi_k}(s^k)\|$.	
	Using \eqref{eq:vr-component-bound}, we obtain
	\begin{align*}
		(\lambda_k-\lambda_{k+1})_+
		\|F_{\xi_k}(t^{k+\frac{1}{2}})-F_{\xi_k}(s^k)\|
		\le&\sqrt{(\lambda_k-\lambda_{k+1})_+}\sqrt{\lambda_k}
		\|F_{\xi_k}(t^{k+\frac{1}{2}})-F_{\xi_k}(s^k)\|\\
		\le&\sqrt{\frac{\beta\bar M}{2}}\sqrt{(\lambda_k-\lambda_{k+1})_+},
	\end{align*}
	which together with \eqref{eq:vr-prediction-gap} and \eqref{f1}  gives
	\begin{equation}\label{eq:threshold-bound}
		\|t^{k+1}-s^k\|\le(1+\beta)R+
		\sqrt{\frac{\beta\bar M}{2}}\,
		\sqrt{(\lambda_k-\lambda_{k+1})_+},
		\qquad k\ge1.
	\end{equation}
	Condition \eqref{eq:w-condition} is equivalent to $w>(1+\beta)R$. Hence, on $A_k$,
	\begin{equation*}
		(\lambda_k-\lambda_{k+1})_+>\frac{2[w-(1+\beta)R]^2}{\beta\bar M}>0
	\end{equation*}
	when $\bar M>0$. Therefore,
	\begin{equation}
		\mathbf 1_{A_k}\le\frac{\beta\bar M}{2[w-(1+\beta)R]^2}
		(\lambda_k-\lambda_{k+1})_+.
		\label{eq:Ak-lambda-bound}
	\end{equation}
	Lemma~\ref{lem41} implies
	\begin{equation*}
		\sum_{k=1}^{\infty}\mathbf 1_{A_k}<\infty
		\qquad\text{almost surely}.
	\end{equation*}
	If $\bar M=0$, then \eqref{eq:threshold-bound} and
	\eqref{eq:w-condition} imply directly that $A_k$ does not occur for
	any $k\ge1$. In either case, the possible event at the single initial
	index $k=0$ does not affect the conclusion.
\end{proof}

\begin{theorem}\label{th41}
	Suppose Assumptions~\ref{ass1}--\ref{ass4} hold, let the iterates be
	generated by Algorithm~\ref{alg4}, and assume \eqref{eq:w-condition} holds.
	Then the following statements hold.
	\begin{enumerate}
		\item[$\mathbf{(i)}$]
		For every $t^*\in U^*$, the sequence
		\begin{equation*}
			\Phi_k(t^*):=\beta\|t^k-t^*\|^2
			+\frac{1-\beta}{P_1}\|s^k-t^*\|^2
		\end{equation*}
		is bounded and converges almost surely.
		
		\item[$\mathbf{(ii)}$]
		The stepsizes are bounded away from zero almost surely: \begin{equation}\label{eq:vr-lambda-lower}
			\lambda_k\ge\underline{\lambda}>0,
			\qquad k\ge0,
		\end{equation}
		where
		\begin{equation}\label{eq:vr-lambda-lower-def}
			\underline{\lambda}=
			\begin{cases}
				\displaystyle
				\min\left\{\lambda_0,\,\lambda_1,\,
				\frac{\beta}{K_0+K_1H_1^\alpha+K_2(\beta w+\sqrt{\bar\lambda})^{\frac{\alpha}{1-\alpha}}},\,
				\frac{1}{H_1^2}\right\},
				& \alpha\in(0,1),\\[4mm]
				\displaystyle
				\min\left\{\lambda_0,\,\lambda_1,\,
				\frac{\beta}{(L_0+L_1H_1)
					e^{L_1(\beta w+\sqrt{\bar\lambda})}},\,
				\frac{1}{H_1^2}\right\},
				& \alpha=1.
			\end{cases}
		\end{equation}
		where $H_1:=\sup_{k\ge0}\max_{i\in[N]}\|F_i(s^k)\|<\infty$  almost surely, and $1/H_1^2=+\infty$ if $H_1=0$; any candidate with a zero denominator is interpreted as
		$+\infty$.

		\item[$\mathbf{(iii)}$]
		There exists a $U^*$-valued random variable $\bar t$ such that
		\begin{equation}\label{eq:vr-as-convergence}
			t^k\to\bar t,
			\qquad
			s^k\to\bar t
			\qquad\text{almost surely}.
		\end{equation}
	\end{enumerate}
\end{theorem}

\begin{proof}
	$\mathbf{(i)}$ Conditioning on $\mathcal G_k$, both $t^{k+1}$ and $A_k$ are
	measurable. If $A_k$ occurs, then $s^{k+1}=t^{k+1}$. If $A_k^c$
	occurs, then $s^{k+1}=t^{k+1}$ with probability $P_1$ and
	$s^{k+1}=s^k$ with probability $1-P_1$. Hence, for $k\ge1$,
	\begin{align}
		\mathbb E_k^+\|s^{k+1}-t^*\|^2
		=&\mathbf 1_{A_k}\|t^{k+1}-t^*\|^2
		+\mathbf 1_{A_k^c}
		\left[P_1\|t^{k+1}-t^*\|^2
		+(1-P_1)\|s^k-t^*\|^2\right].
		\label{eq:snapshot-cond-exp}
	\end{align}
	Rearranging \eqref{eq:snapshot-cond-exp} gives
	\begin{align}
		\frac{1-\beta}{P_1}
		\mathbb E_k^+\|s^{k+1}-t^*\|^2
		={}&(1-\beta)\|t^{k+1}-t^*\|^2
		+(1-\beta)\left(\frac1{P_1}-1\right)\|s^k-t^*\|^2
		\notag\\
		&+(1-\beta)
		\left(\frac1{P_1}-1\right)\mathbf 1_{A_k}
		\left(\|t^{k+1}-t^*\|^2-\|s^k-t^*\|^2\right).
		\label{eq:loopless-snapshot-identity}
	\end{align}
	
	We next bound the last term in
	\eqref{eq:loopless-snapshot-identity}. By the difference-of-squares
	identity and Young's inequality,
	\begin{align}
		\|t^{k+1}-t^*\|^2-\|s^k-t^*\|^2
		={}&2\langle s^k-t^*,t^{k+1}-s^k\rangle
		+\|t^{k+1}-s^k\|^2
		\notag\\
		\le{}&\|s^k-t^*\|^2+2\|t^{k+1}-s^k\|^2.
		\label{eq:snapshot-distance-difference}
	\end{align}
	Moreover, \eqref{eq:threshold-bound} and
	$(\lambda_k-\lambda_{k+1})_+\le\lambda_k\le\bar\lambda$ imply
	\begin{equation}\label{eq:tk1-sk-uniform}
		\|t^{k+1}-s^k\|\le(1+\beta)R+
		\sqrt{\frac{\beta\bar M\bar\lambda}{2}}
		=:\widehat C,
		\qquad k\ge1.
	\end{equation}
	Since
	\begin{equation*}
		\|s^k-t^*\|^2\le\frac{P_1}{1-\beta}\Phi_k(t^*),
	\end{equation*}
	it follows from \eqref{eq:snapshot-distance-difference} and
	\eqref{eq:tk1-sk-uniform} that
	\begin{align}
		&(1-\beta)
		\left(\frac1{P_1}-1\right)
		\mathbf 1_{A_k}
		\left(\|t^{k+1}-t^*\|^2-\|s^k-t^*\|^2\right)
		\notag\\
		&\qquad\le
		(1-P_1)\mathbf 1_{A_k}\Phi_k(t^*)
		+2(1-\beta)
		\left(\frac1{P_1}-1\right)
		\widehat C^2\mathbf 1_{A_k}.
		\label{eq:snapshot-extra-bound}
	\end{align}
	
	By the definition of $\Phi_{k+1}(t^*)$,
	\eqref{eq:loopless-snapshot-identity}, and
	\eqref{eq:snapshot-extra-bound},
	\begin{align}
		\mathbb E_k^+\Phi_{k+1}(t^*)
		={}&\beta\|t^{k+1}-t^*\|^2
		+\frac{1-\beta}{P_1}
		\mathbb E_k^+\|s^{k+1}-t^*\|^2
		\notag\\
		\le{}&\|t^{k+1}-t^*\|^2
		+(1-\beta)\left(\frac1{P_1}-1\right)\|s^k-t^*\|^2
		\notag\\
		&+(1-P_1)\mathbf 1_{A_k}\Phi_k(t^*)
		+2(1-\beta)\left(\frac1{P_1}-1\right)
\widehat C^2	\mathbf 1_{A_k}.
		\label{eq:Phi-Gk-bound}
	\end{align}
	Taking conditional expectation with respect to $\mathcal F_k$ and using
	the tower property, together with the $\mathcal F_k$-measurability of
	$\Phi_k(t^*)$, yields
	\begin{align}
		\mathbb E_k\Phi_{k+1}(t^*)
		\le{}&\mathbb E_k\|t^{k+1}-t^*\|^2
		+(1-\beta)\left(\frac1{P_1}-1\right)\|s^k-t^*\|^2
		\notag\\
		&+(1-P_1)\mathbb E_k\mathbf 1_{A_k}\Phi_k(t^*)
		+2(1-\beta)\left(\frac1{P_1}-1\right)
		\widehat C^2\mathbb E_k\mathbf 1_{A_k},
		\label{eq:Phi-Fk-bound}
	\end{align}
	Applying Lemma~\ref{lem42} to the first term on the right-hand side of
	\eqref{eq:Phi-Fk-bound}, and using
	\begin{equation*}
		(1-\beta)+(1-\beta)\left(\frac1{P_1}-1\right)
		=\frac{1-\beta}{P_1},
	\end{equation*}
	we obtain
	\begin{align}
		\mathbb E_k\Phi_{k+1}(t^*)
		\le{}&
		\left[1+(1-P_1)\mathbb E_k\mathbf 1_{A_k}\right]\Phi_k(t^*)
		-2\mu\lambda_k
		\operatorname{dist}^p(t^{k+\frac12},U^*)
		\notag\\
		&-(1-\beta)\mathbb E_k\|t^{k+1}-t^{k+\frac12}\|^2
		-(1-2\beta)\|s^k-t^{k+\frac12}\|^2
		-\beta\|t^k-t^{k+\frac12}\|^2
		\notag\\
		&+\bar M\mathbb E_k(\lambda_k-\lambda_{k+1})_+
		+2(1-\beta)\left(\frac1{P_1}-1\right)
	\widehat C^2\mathbb E_k\mathbf 1_{A_k},
		\qquad k\ge1
		\label{eq:vr-global-quasi-fejer},
	\end{align}
	
	It remains to verify the summability of the error terms in
	\eqref{eq:vr-global-quasi-fejer}. By the proof of
	Lemma~\ref{lem:finite-threshold}, \eqref{eq:Ak-lambda-bound} holds.
		Taking conditional expectation in \eqref{eq:Ak-lambda-bound} gives
	\begin{equation}\label{eq:Ak-conditional-bound}
		\mathbb E_k\mathbf 1_{A_k}
		\le\frac{\beta\bar M}{2[w-(1+\beta)R]^2}
		\mathbb E_k(\lambda_k-\lambda_{k+1})_+.
	\end{equation}
By Tonelli's theorem, the tower property, and Lemma~\ref{lem41}, we get
	\begin{align}
		\mathbb E
		\left[\sum_{k=1}^{\infty}
		\mathbb E_k(\lambda_k-\lambda_{k+1})_+\right]
		&=\sum_{k=1}^{\infty}
		\mathbb E(\lambda_k-\lambda_{k+1})_+<\infty.
		\label{eq:conditional-lambda-sum}
	\end{align}
	Therefore,
	\begin{equation*}\label{eq:conditional-lambda-summable}
		\sum_{k=1}^{\infty}
		\mathbb E_k(\lambda_k-\lambda_{k+1})_+
		<\infty
		\qquad\text{almost surely}.
	\end{equation*}
	Together with \eqref{eq:Ak-conditional-bound}, this further gives
	\begin{equation*}\label{eq:Ak-conditional-summable}
		\sum_{k=1}^{\infty}
		\mathbb E_k\mathbf 1_{A_k}<\infty
		\qquad\text{almost surely}.
	\end{equation*}
	Hence both the multiplicative error and the additive error in
	\eqref{eq:vr-global-quasi-fejer} are summable almost surely.
	
	Since $\beta\in\left(0,\frac12\right)$,
	all the descent terms in \eqref{eq:vr-global-quasi-fejer} are
	nonnegative. Applying Lemma~\ref{lem:rs} to the tail sequence
	$\{\Phi_k(t^*)\}_{k\ge1}$, we conclude that
	\begin{equation*}\label{eq:Phi-converges}
		\Phi_k(t^*)
		\quad\text{converges almost surely}.
	\end{equation*}
	Moreover, 
	\begin{equation}\label{eq:vr-tail-summability}
		\sum_{k=1}^{\infty}
		\lambda_k\operatorname{dist}^p(t^{k+\frac12},U^*)<\infty,
		\qquad
		\sum_{k=1}^{\infty}\|s^k-t^{k+\frac12}\|^2<\infty,
		\qquad
		\sum_{k=1}^{\infty}\|t^k-t^{k+\frac12}\|^2<\infty
	\end{equation}
	almost surely.
	
	Finally, convergence of $\Phi_k(t^*)$ implies that
	$\{\Phi_k(t^*)\}$ is bounded almost surely. Since
	$\beta>0$ and $(1-\beta)/P_1>0$, the definition of
	$\Phi_k(t^*)$ then implies that both $\{t^k\}$ and
	$\{s^k\}$ are bounded almost surely. This proves
	part~$\mathbf{(i)}$.

	\medskip
	\noindent\textbf{(ii)}
	We next establish a positive lower bound for the stepsizes. By
	part~$\mathbf{(i)}$, $\{s^k\}$ is bounded almost surely. Since
	$U$ is closed, $N<\infty$, and every $F_i$ is continuous, it follows
	that
	\begin{equation}\label{eq:H1-finite}
		H_1:=\sup_{k\ge0}
		\max_{i\in[N]}\|F_i(s^k)\|<\infty
		\qquad\text{almost surely}.
	\end{equation}
	
	Suppose first that $\alpha\in(0,1)$. By
	Assumption~\ref{ass4}, Proposition~\ref{pro3.1},
	\eqref{eq:H1-finite}, and \eqref{eq:vr-prediction-gap},
	\begin{align*}
		\|F_{\xi_k}(t^{k+\frac12})-F_{\xi_k}(s^k)\|
		&\le\|t^{k+\frac12}-s^k\|
		\left(K_0+K_1\|F_{\xi_k}(s^k)\|^\alpha+K_2
		\|t^{k+\frac12}-s^k\|^{\frac{\alpha}{1-\alpha}}
		\right)
		\notag\\
		&\le\|t^{k+\frac12}-s^k\|
		\left(K_0+K_1H_1^\alpha+K_2R^{\frac{\alpha}{1-\alpha}}\right),
		\qquad k\ge1.
		\label{eq:vr-component-lower-alpha}
	\end{align*}
	Hence, whenever
	$F_{\xi_k}(t^{k+\frac12})\ne F_{\xi_k}(s^k)$,
	the second candidate in \eqref{eq:aeg-vr-stepsize} satisfies
	\begin{equation}\label{eq:vr-local-candidate-alpha}
		\frac{\beta\|t^{k+\frac12}-s^k\|}
		{\|F_{\xi_k}(t^{k+\frac12})-F_{\xi_k}(s^k)\|}
		\ge\frac{\beta}{K_0+K_1H_1^\alpha+K_2R^{\frac{\alpha}{1-\alpha}}}.
	\end{equation}
	If the denominator is zero, the corresponding candidate is interpreted as
	$+\infty$, so the same lower bound remains valid.
	
	If $\alpha=1$, Proposition~\ref{pro3.1},
	\eqref{eq:H1-finite}, and \eqref{eq:vr-prediction-gap} give
	\begin{align*}
		\|F_{\xi_k}(t^{k+\frac12})-F_{\xi_k}(s^k)\|
		&\le\|t^{k+\frac12}-s^k\|
		\bigl(L_0+L_1\|F_{\xi_k}(s^k)\|\bigr)
		\exp\left(L_1\|t^{k+\frac12}-s^k\|\right)
		\notag\\
		&\le\|t^{k+\frac12}-s^k\|(L_0+L_1H_1)e^{L_1R},
		\qquad k\ge1.
		\label{eq:vr-component-lower-one}
	\end{align*}
	Therefore, whenever the denominator is nonzero,
	\begin{equation}\label{eq:vr-local-candidate-one}
		\frac{\beta\|t^{k+\frac12}-s^k\|}
		{\|F_{\xi_k}(t^{k+\frac12})-F_{\xi_k}(s^k)\|}
		\ge\frac{\beta}{(L_0+L_1H_1)e^{L_1R}}.
	\end{equation}

	By \eqref{eq:H1-finite},
	\begin{equation*}
		\underline{\lambda}>0
		\qquad\text{almost surely}.
	\end{equation*}
	Moreover, $\lambda_0\ge\underline{\lambda}$ and $\lambda_1\ge\underline{\lambda}$.	Suppose that
	$\lambda_k\ge\underline{\lambda}$ for some $k\ge1$.
	Then the first candidate in \eqref{eq:aeg-vr-stepsize} satisfies $(1+\eta_k)\lambda_k\ge\lambda_k\ge\underline{\lambda}$.
	By \eqref{eq:vr-local-candidate-alpha} or
	\eqref{eq:vr-local-candidate-one}, the second candidate is also no smaller
	than $\underline{\lambda}$. Finally,
	\begin{equation*}
		\frac{1}{\max_{i\in[N]}\|F_i(s^{k+1})\|^2}
		\ge\frac1{H_1^2}
		\ge\underline{\lambda}.
	\end{equation*}
	Thus every candidate in \eqref{eq:aeg-vr-stepsize} is bounded below by
	$\underline{\lambda}$, and hence $\lambda_{k+1}\ge\underline{\lambda}$.
	By induction, \eqref{eq:vr-lambda-lower} holds almost surely. This proves part~$\mathbf{(ii)}$.

	\medskip
	\noindent\textbf{(iii)}
	We finally prove the convergence of the iterates. By
	\eqref{eq:vr-tail-summability} and
	\eqref{eq:vr-lambda-lower},
	\begin{equation*}
		\sum_{k=1}^{\infty}
		\operatorname{dist}^p(t^{k+\frac12},U^*)
		\le\frac1{\underline{\lambda}}
		\sum_{k=1}^{\infty}
		\lambda_k\operatorname{dist}^p(t^{k+\frac12},U^*)
		<\infty
	\end{equation*}
	almost surely. Hence
	\begin{equation}\label{eq:vr-residual-limits}
		\operatorname{dist}(t^{k+\frac12},U^*)\to0,
		\qquad
		\|s^k-t^{k+\frac12}\|\to0,
		\qquad
		\|t^k-t^{k+\frac12}\|\to0
	\end{equation}
	almost surely, where the last two limits follow directly from
	\eqref{eq:vr-tail-summability}. Consequently,
	\begin{equation}\label{eq:tk-sk-zero}
		\|t^k-s^k\|\le\|t^k-t^{k+\frac12}\|
		+\|t^{k+\frac12}-s^k\|\to0
		\qquad\text{almost surely}.
	\end{equation}
	
	We next show that every accumulation point of $\{t^k\}$ belongs to
	$U^*$. Fix a sample path on which the preceding properties hold.
	By part~$\mathbf{(i)}$, $\{t^k\}$ is bounded. Let
	$\{t^{k_j}\}$ be a convergent subsequence such that $t^{k_j}\to u$. 
	By \eqref{eq:vr-residual-limits}, $\|t^{k_j+\frac12}-t^{k_j}\|\to0$, and hence $t^{k_j+\frac12}\to u$.	Since the distance function is continuous,
	\eqref{eq:vr-residual-limits} yields
	\begin{equation*}
		\operatorname{dist}(u,U^*)
		=\lim_{j\to\infty}
		\operatorname{dist}(t^{k_j+\frac12},U^*)=0.
	\end{equation*}
	Since $U^*$ is closed, $u\in U^*$. Thus every accumulation point of
	$\{t^k\}$ belongs to $U^*$.
	
	It remains to show that the accumulation point is unique. Since
	$U^*\subseteq\mathbb R^m$ is separable, let
	$\{u_j^*\}_{j\ge1}$ be a countable dense subset of  $U^*$.
	By part~$\mathbf{(i)}$, for every $j\ge1$,
	$\Phi_k(u_j^*)$ converges almost surely. Hence there exists an event of
	probability one on which
	\begin{equation*}\label{eq:Phi-dense-convergence}
		\Phi_k(u_j^*)
		\quad\text{converges for every }j\ge1.
	\end{equation*}
	Fix a sample path in this event, and let $u',u''\in U^*$ be two
	accumulation points of $\{t^k\}$. Then there exist subsequences such that
	\begin{equation*}
		t^{k_\ell}\to u',
		\qquad
		t^{m_\ell}\to u''.
	\end{equation*}
	By \eqref{eq:tk-sk-zero}, we have
	\begin{equation*}
		s^{k_\ell}\to u',
		\qquad
		s^{m_\ell}\to u''.
	\end{equation*}
	For every $j\ge1$, the convergence of
	$\Phi_k(u_j^*)$ implies that its limits along these two subsequences
	coincide. Therefore,
	\begin{equation}\label{eq:cluster-equality}
		\left(\beta+\frac{1-\beta}{P_1}\right)\|u'-u_j^*\|^2
		=\left(\beta+\frac{1-\beta}{P_1}\right)\|u''-u_j^*\|^2,
		\qquad j\ge1.
	\end{equation}
	By density, there exists a subsequence
	$\{u_{j_\ell}^*\}$ such that $u_{j_\ell}^*\to u'$. Passing to the limit in \eqref{eq:cluster-equality} along this subsequence
	gives
	\begin{equation*}
		\left(\beta+\frac{1-\beta}{P_1}\right)\|u''-u'\|^2=0.
	\end{equation*}
	Since $\beta+\frac{1-\beta}{P_1}>0$, we obtain $u''=u'$. Thus $\{t^k\}$ has a unique accumulation point.	Since it is bounded, there exists $\bar t\in U^*$ such that
	\begin{equation*}
		t^k\to\bar t
		\qquad\text{almost surely}.
	\end{equation*}
	Finally, \eqref{eq:tk-sk-zero} implies
	\begin{equation*}
		s^k\to\bar t
		\qquad\text{almost surely}.
	\end{equation*}
	This proves \eqref{eq:vr-as-convergence} and part~$\mathbf{(iii)}$.
\end{proof}

\begin{lemma}[Eventual stability]\label{lem:vr-eventual-stability}
	Under the assumptions of Theorem~\ref{th41}, there exists an almost surely
	finite random index $K$ such that
	\begin{equation}\label{eq:vr-eventual-stability}
		\|t^{k+1}-s^k\|\le w,
		\qquad
		1-\beta-\beta\frac{\lambda_k^2}{\lambda_{k+1}^2}
		\ge\frac{1-2\beta}{2},
		\qquad k\ge K,
	\end{equation}
	almost surely. Moreover,
	\begin{equation}\label{eq:vr-prediction-distance-bound}
		\sup_{k\ge0}\operatorname{dist}(t^{k+\frac12},U^*)<\infty
		\qquad\text{almost surely}.
	\end{equation}
\end{lemma}

\begin{proof}
	By Lemma~\ref{lem:finite-threshold}, the event $A_k$ occurs only finitely many times almost surely. Hence there exists an almost surely finite random index $K_1$ such that
	\begin{equation*}\label{eq:eventual-snapshot-gap}
		\|t^{k+1}-s^k\|\le w,
		\qquad k\ge K_1,
	\end{equation*}
	almost surely. By Lemma~\ref{lem41} and Theorem~\ref{th41}(ii), the stepsize sequence converges almost surely: $\lambda_k\to\lambda_\infty\ge	\underline{\lambda}>0$.	It follows that $\frac{\lambda_k}{\lambda_{k+1}}\to1$ almost surely. Consequently,
	\begin{equation*}
		1-\beta-\beta\frac{\lambda_k^2}{\lambda_{k+1}^2}
		\to1-2\beta
		\qquad\text{almost surely}.
	\end{equation*}
	Since $\beta\in(0,\frac12)$, we have $1-2\beta>0$. Thus there exists
	an almost surely finite random index $K_2$ such that
	\begin{equation*}\label{eq:eventual-ratio-bound}
		1-\beta-\beta\frac{\lambda_k^2}{\lambda_{k+1}^2}
		\ge\frac{1-2\beta}{2},
		\qquad k\ge K_2,
	\end{equation*}
	almost surely. Taking $K=\max\{K_1,K_2\}$ gives \eqref{eq:vr-eventual-stability}.
	
	It remains to prove \eqref{eq:vr-prediction-distance-bound}.
	By Theorem~\ref{th41}(iii), there exists
	$\bar t\in U^*$ such that $s^k\to\bar t$ almost surely.	Hence
	\begin{equation*}
		\sup_{k\ge0}\|s^k-\bar t\|<\infty
		\qquad\text{almost surely}.
	\end{equation*}
	For every $k\ge1$, \eqref{eq:vr-prediction-gap} gives
	\begin{align*}
		\operatorname{dist}(t^{k+\frac12},U^*)
		&\le\|t^{k+\frac12}-\bar t\|
		\le\|t^{k+\frac12}-s^k\|+\|s^k-\bar t\|
		\le R+\|s^k-\bar t\|.
	\end{align*}
	Therefore,
	\begin{equation*}
		\sup_{k\ge1}\operatorname{dist}(t^{k+\frac12},U^*)<\infty
		\qquad\text{almost surely}.
	\end{equation*}
	Since $t^{1/2}$ is finite, \eqref{eq:vr-prediction-distance-bound}
	follows.
\end{proof}

For the expected-rate analysis, we strengthen the almost-sure bounds obtained
in Theorem~\ref{th41}(ii) and Lemma~\ref{lem:vr-eventual-stability} by
requiring them to hold uniformly with deterministic constants. Specifically,
assume that there exist $k_{\rm r}\ge1$ and $\lambda_{\min}>0$ such that
\begin{equation}\label{eq:vr-rate-uniformity}
	K\le k_{\rm r},
	\qquad
	\underline{\lambda}\ge\lambda_{\min}
	\qquad\text{almost surely}.
\end{equation}
When $0<p<2$, assume additionally that there exists a deterministic
constant $B>0$ such that
\begin{equation}\label{eq:vr-uniform-B}
	\sup_{k\ge0}\operatorname{dist}(t^{k+\frac12},U^*)\le B
	\qquad\text{almost surely}.
\end{equation}
No additional boundedness condition is needed when $p=2$.

\begin{theorem}\label{th43}
	Suppose the assumptions of Theorem~\ref{th41} and
	\eqref{eq:vr-rate-uniformity} hold.
	\begin{enumerate}
		\item[$\mathbf{(i)}$]
		If $0<p\le2$, and \eqref{eq:vr-uniform-B} additionally holds when
		$0<p<2$, then there exist $c>0$ such that
		\begin{equation}\label{eq:vr-linear-rate}
			\mathbb E\left[\inf_{u^*\in U^*}\Phi_k(u^*)\right]
			\le c\rho^k,
			\qquad k\ge0,
		\end{equation}
		where
		\begin{equation}\label{eq:vr-rho}
			\rho:=1-
			\min\left\{\frac12,\,
			\frac{(1-2\beta)P_1}{4(1-\beta)},\,
			\frac{\mu\lambda_{\min}}{\left(\beta+\frac{1-\beta}{P_1}\right)B^{2-p}}
			\right\}\in(0,1).
		\end{equation}
		
		\item[$\mathbf{(ii)}$]
		If $p>2$, then there exists $c_p>0$ such that
		\begin{equation}\label{eq:vr-sublinear-rate}
			\mathbb E\left[\inf_{u^*\in U^*}\Phi_k(u^*)\right]
			\le\frac{c_p}{(k+1)^{2/(p-2)}},
			\qquad k\ge0.
		\end{equation}
	\end{enumerate}
\end{theorem}

\begin{proof}
	By \eqref{eq:vr-rate-uniformity},
	Theorem~\ref{th41}(ii), and
	Lemma~\ref{lem:vr-eventual-stability}, for every
	$k\ge k_{\rm r}$, almost surely,
	\begin{equation}\label{eq:vr-rate-tail-bounds}
		\lambda_k\ge\lambda_{\min},
		\qquad
		\|t^{k+1}-s^k\|\le w,
		\qquad
		1-\beta-\beta\frac{\lambda_k^2}{\lambda_{k+1}^2}
		\ge\frac{1-2\beta}{2}.
	\end{equation}
	Fix $k\ge k_{\rm r}$. The second inequality in
	\eqref{eq:vr-rate-tail-bounds} implies that the snapshot is updated
	according to the Bernoulli rule. Hence, for every $u^*\in U^*$,
	\begin{equation}\label{eq:vr-rate-snapshot}
		\frac{1-\beta}{P_1}
		\mathbb E_k^+\|s^{k+1}-u^*\|^2
		=(1-\beta)\|t^{k+1}-u^*\|^2
		+(1-\beta)\left(\frac1{P_1}-1\right)\|s^k-u^*\|^2.
	\end{equation}
	By the second candidate in \eqref{eq:aeg-vr-stepsize},
	\begin{equation}\label{eq:vr-rate-local-control}
		\frac{\lambda_k^2}{\beta}
		\|F_{\xi_k}(t^{k+\frac12})-F_{\xi_k}(s^k)\|^2
		\le\beta\frac{\lambda_k^2}{\lambda_{k+1}^2}
		\|t^{k+\frac12}-s^k\|^2.
	\end{equation}
	If the operator difference vanishes, this inequality holds trivially.
	Substituting \eqref{eq:vr-rate-local-control} into
	\eqref{eq:vr-prelocal} gives
	\begin{align}
		\|t^{k+1}-u^*\|^2
		\le{}&
		\beta\|t^k-u^*\|^2
		+(1-\beta)\|s^k-u^*\|^2
		-2\lambda_k
		\langle t^{k+\frac12}-u^*,\widehat F_k\rangle
		\notag\\
		&-(1-\beta)\|t^{k+1}-t^{k+\frac12}\|^2
		-\left(1-\beta-\beta\frac{\lambda_k^2}{\lambda_{k+1}^2}\right)
		\|s^k-t^{k+\frac12}\|^2
		\notag\\
		&-\beta\|t^k-t^{k+\frac12}\|^2.
		\label{eq:vr-rate-pre}
	\end{align}
	Using \eqref{eq:vr-rate-tail-bounds} in
	\eqref{eq:vr-rate-pre} and dropping the nonpositive term involving
	$\|t^{k+1}-t^{k+\frac12}\|^2$, we obtain
	\begin{align*}
		\|t^{k+1}-u^*\|^2
		\le{}&
		\beta\|t^k-u^*\|^2
		+(1-\beta)\|s^k-u^*\|^2
		-2\lambda_k
		\langle t^{k+\frac12}-u^*,\widehat F_k\rangle
		\notag\\
		&-
		\frac{1-2\beta}{2}
		\|s^k-t^{k+\frac12}\|^2
		-\beta\|t^k-t^{k+\frac12}\|^2.
	\end{align*}
	Taking conditional expectation with respect to $\mathcal F_k$, using
	\eqref{eq:svrg-unbiased}, $p$-quasi-sharpness, and
	$\lambda_k\ge\lambda_{\min}$, gives
	\begin{align}
		\mathbb E_k\|t^{k+1}-u^*\|^2
		\le{}&
		\beta\|t^k-u^*\|^2
		+(1-\beta)\|s^k-u^*\|^2
	-2\mu\lambda_{\min}
		\operatorname{dist}^p(t^{k+\frac12},U^*)
		\notag\\
		&-\frac{1-2\beta}{2}
		\|s^k-t^{k+\frac12}\|^2
		-\beta\|t^k-t^{k+\frac12}\|^2.
		\label{eq:vr-rate-t}
	\end{align}
	Taking conditional expectation with respect to $\mathcal F_k$ in
	\eqref{eq:vr-rate-snapshot} and using the tower property, we have
	\begin{align*}
		&\mathbb E_k
		\left[\beta\|t^{k+1}-u^*\|^2
		+\frac{1-\beta}{P_1}
		\|s^{k+1}-u^*\|^2\right]
		=\mathbb E_k\|t^{k+1}-u^*\|^2
		+(1-\beta)\left(\frac1{P_1}-1\right)
		\|s^k-u^*\|^2.
	\end{align*}
	Combining this identity with \eqref{eq:vr-rate-t} yields
	\begin{align*}
		\mathbb E_k[\Phi_{k+1}(u^*)]
		\le&
		\Phi_k(u^*)
		-2\mu\lambda_{\min}
		\operatorname{dist}^p(t^{k+\frac12},U^*)
		-\frac{1-2\beta}{2}
		\|s^k-t^{k+\frac12}\|^2
		-\beta\|t^k-t^{k+\frac12}\|^2.
	\end{align*}
	Since the preceding inequality holds for every $u^*\in U^*$, and
	\begin{equation*}
		\mathbb E_k\left[
		\inf_{u^*\in U^*}\Phi_{k+1}(u^*)\right]
		\le\inf_{u^*\in U^*}
		\mathbb E_k[\Phi_{k+1}(u^*)],
	\end{equation*}
	we obtain
	\begin{align}
		\mathbb E_k
		\left[\inf_{u^*\in U^*}\Phi_{k+1}(u^*)\right]
		\le{}&
		\inf_{u^*\in U^*}\Phi_k(u^*)
		-2\mu\lambda_{\min}
		\operatorname{dist}^p(t^{k+\frac12},U^*)
		\notag\\
		&-\frac{1-2\beta}{2}
		\|s^k-t^{k+\frac12}\|^2
		-\beta\|t^k-t^{k+\frac12}\|^2,
		\qquad k\ge k_{\rm r}.
		\label{eq:vr-rate-descent}
	\end{align}
	
	\medskip
	\noindent\textbf{(i)}
	Suppose $0<p\le2$. For each $k\ge k_{\rm r}$, let
	$u_k^*\in U^*$ satisfy $\|t^{k+\frac12}-u_k^*\|=\operatorname{dist}(t^{k+\frac12},U^*)$. 	By the definition of $\Phi_k$ and	$\|a+b\|^2\le2\|a\|^2+2\|b\|^2$,
	\begin{align}
		\inf_{u^*\in U^*}\Phi_k(u^*)
		\le{}&
		2\beta\|t^k-t^{k+\frac12}\|^2
		+\frac{2(1-\beta)}{P_1}
		\|s^k-t^{k+\frac12}\|^2
		\notag\\
		&+2\left(\beta+\frac{1-\beta}{P_1}\right)
		\operatorname{dist}^2(t^{k+\frac12},U^*).
		\label{eq:vr-linear-distance-bound}
	\end{align}
	If $0<p<2$, \eqref{eq:vr-uniform-B} implies
	\begin{equation*}
		\operatorname{dist}^2(t^{k+\frac12},U^*)
		\le B^{2-p}
		\operatorname{dist}^p(t^{k+\frac12},U^*).
	\end{equation*}
	For $p=2$, the same relation holds with $B^{2-p}=1$. 
	Combining the definition of $\rho$ in \eqref{eq:vr-rho} and
	\eqref{eq:vr-linear-distance-bound} gives
	\begin{align*}
		(1-\rho)
		\inf_{u^*\in U^*}\Phi_k(u^*)
		\le{}&
		\beta\|t^k-t^{k+\frac12}\|^2
		+\frac{1-2\beta}{2}
		\|s^k-t^{k+\frac12}\|^2
		+2\mu\lambda_{\min}
		\operatorname{dist}^p(t^{k+\frac12},U^*),
	\end{align*}
	which together with
	\eqref{eq:vr-rate-descent} yields
	\begin{equation*}\label{eq:vr-linear-contraction}
		\mathbb E_k
		\left[\inf_{u^*\in U^*}\Phi_{k+1}(u^*)\right]
		\le\rho\inf_{u^*\in U^*}\Phi_k(u^*),
		\qquad k\ge k_{\rm r}.
	\end{equation*}
	Taking total expectations and iterating gives
	\begin{equation*}\label{eq:vr-linear-tail}
		\mathbb E
		\left[\inf_{u^*\in U^*}\Phi_k(u^*)\right]
		\le\rho^{k-k_{\rm r}}
		\mathbb E
		\left[\inf_{u^*\in U^*}\Phi_{k_{\rm r}}(u^*)\right],
		\qquad k\ge k_{\rm r}.
	\end{equation*}
	Hence \eqref{eq:vr-linear-rate} holds for every $k\ge0$ by taking
	\begin{equation*}
		c:=\max_{0\le j\le k_{\rm r}}
		\rho^{-j}
		\mathbb E
		\left[\inf_{u^*\in U^*}\Phi_j(u^*)\right].
	\end{equation*}

	\medskip
	\noindent\textbf{(ii)}
	Suppose $p>2$. 
	Using the same argument as in \eqref{eq:vr-prediction-gap}, we obtain
	\begin{align*}
		\|t^k-t^{k+\frac12}\|
		&=\left\|P_U(t^k)
		-P_U\!\left(\bar t^k-\lambda_kF(s^k)\right)\right\|\\
		&\le(1-\beta)\|t^k-s^k\|+\lambda_k\|F(s^k)\|\\
		&\le(1-\beta)w+\sqrt{\bar\lambda},
		\qquad k\ge1,
	\end{align*}
	which together with \eqref{eq:vr-prediction-gap} gives
	\begin{equation}\label{eq:vr-weighted-gap-bound}
		\beta\|t^k-t^{k+\frac12}\|^2
		+\frac{1-\beta}{P_1}
		\|s^k-t^{k+\frac12}\|^2
		\le	R_0^2,
		\qquad k\ge k_{\rm r},
	\end{equation}
	where
		\begin{equation*}
		R_0^2:=\beta
		\left[(1-\beta)w+\sqrt{\bar\lambda}\right]^2
		+\frac{1-\beta}{P_1}R^2.
	\end{equation*}
	Let $u_k^*\in U^*$ be a nearest point to
	$t^{k+\frac12}$. By the triangle inequality,
	\begin{align*}
		\Phi_k(u_k^*)^{1/2}
		\le{}&
		\left[\beta\|t^k-t^{k+\frac12}\|^2
		+\frac{1-\beta}{P_1}
		\|s^k-t^{k+\frac12}\|^2\right]^{1/2}+
		\left(\beta+\frac{1-\beta}{P_1}\right)^{1/2}
		\operatorname{dist}(t^{k+\frac12},U^*).
	\end{align*}
	Using $(a+b)^p\le2^{p-1}(a^p+b^p)$ and
	\eqref{eq:vr-weighted-gap-bound}, we obtain
	\begin{align*}
		\left(\inf_{u^*\in U^*}\Phi_k(u^*)\right)^{p/2}
		\le{}&
		2^{p-1}R_0^{p-2}
		\left[\beta\|t^k-t^{k+\frac12}\|^2
		+\frac{1-\beta}{P_1}\|s^k-t^{k+\frac12}\|^2\right]
		\notag\\
		&+
		2^{p-1}\left(\beta+\frac{1-\beta}{P_1}\right)^{p/2}
		\operatorname{dist}^p(t^{k+\frac12},U^*),
	\end{align*}
	which implies
	\begin{align*}
		a_p\left(\inf_{u^*\in U^*}\Phi_k(u^*)\right)^{p/2}
		\le{}&
		\beta\|t^k-t^{k+\frac12}\|^2
		+\frac{1-2\beta}{2}
		\|s^k-t^{k+\frac12}\|^2
		+2\mu\lambda_{\min}
		\operatorname{dist}^p(t^{k+\frac12},U^*),
		\label{eq:vr-sublinear-control}
	\end{align*}
	where 
	\begin{equation*}\label{eq:vr-ap}
		a_p:=
		\min\left\{
		2^{1-p}R_0^{2-p}
		\min\left\{1,\,
		\frac{(1-2\beta)P_1}{2(1-\beta)}\right\},
		\frac{2^{2-p}\mu\lambda_{\min}}{
			\left(\beta+\frac{1-\beta}{P_1}\right)^{p/2}}
		\right\}.
	\end{equation*}
	Combining this estimate with
	\eqref{eq:vr-rate-descent} gives
	\begin{equation}
		\mathbb E_k
		\left[\inf_{u^*\in U^*}\Phi_{k+1}(u^*)\right]
		\le\inf_{u^*\in U^*}\Phi_k(u^*)
		-a_p\left(\inf_{u^*\in U^*}\Phi_k(u^*)\right)^{p/2},
		\qquad k\ge k_{\rm r}.
		\label{eq:vr-sublinear-recursion}
	\end{equation}
	Taking total expectations in
	\eqref{eq:vr-sublinear-recursion} and applying Jensen's inequality,
	since $p/2>1$, yields
	\begin{equation*}\label{eq:vr-sublinear-expectation-recursion}
		\mathbb E
		\left[\inf_{u^*\in U^*}\Phi_{k+1}(u^*)\right]
		\le\mathbb E
		\left[\inf_{u^*\in U^*}\Phi_k(u^*)\right]-a_p\left(\mathbb E
		\left[\inf_{u^*\in U^*}\Phi_k(u^*)\right]\right)^{p/2},
		\qquad k\ge k_{\rm r}.
	\end{equation*}
	Applying Lemma~\ref{lem25} with $q=(p-2)/2$ gives
	\begin{equation}\label{eq:vr-sublinear-tail}
		\mathbb E\left[\inf_{u^*\in U^*}\Phi_{k}(u^*)\right]
		\le\frac{\mathbb E
			\left[\displaystyle\inf_{u^*\in U^*}\Phi_{k_{\rm r}}(u^*)\right]}
			{\left[1+\frac{p-2}{2}
			a_p\left(\mathbb E
			\left[\displaystyle
			\inf_{u^*\in U^*}\Phi_{k_{\rm r}}(u^*)
			\right]\right)^{(p-2)/2}
			(k-k_{\rm r})\right]^{2/(p-2)}},
		\qquad k\ge k_{\rm r}.
	\end{equation}
	If $\mathbb E\left[\displaystyle\inf_{u^*\in U^*}\Phi_{k_{\rm r}}(u^*)\right]=0$, the result follows immediately. Otherwise, for $k\ge k_{\rm r}+1$,
\begin{equation*}
	\mathbb E\left[\inf_{u^*\in U^*}\Phi_k(u^*)\right]
	\le\left[\frac{p-2}{2}a_p(k-k_{\rm r})\right]^{-2/(p-2)}.
\end{equation*}
Since $k-k_{\rm r}\ge\frac{k+1}{k_{\rm r}+2}$ for $k\ge k_{\rm r}+1$, we obtain
\begin{equation*}
	\mathbb E\left[\inf_{u^*\in U^*}\Phi_k(u^*)\right]
	\le\left[\frac{2(k_{\rm r}+2)}{(p-2)a_p}\right]^{2/(p-2)}
	\frac1{(k+1)^{2/(p-2)}}.
\end{equation*}
Thus \eqref{eq:vr-sublinear-rate} holds for every $k\ge0$ by choosing
\begin{equation*}
	c_p:=\max\left\{
	\max_{0\le j\le k_{\rm r}}(j+1)^{2/(p-2)}
	\mathbb E
	\left[\inf_{u^*\in U^*}\Phi_j(u^*)\right],
	\left[\frac{2(k_{\rm r}+2)}{(p-2)a_p}\right]^{2/(p-2)}
	\right\}.
\end{equation*}
\end{proof}

\begin{remark}
	The rates in Theorem~\ref{th43} concern the expected
	infimum of the Lyapunov function. They also imply the corresponding
	expected squared-distance rates because, for every $k$,
	\begin{equation*}
		\beta\,\operatorname{dist}^2(t^k,U^*)
		\le \inf_{u^*\in U^*}\Phi_k(u^*),
		\qquad
		\frac{1-\beta}{P_1}\,\operatorname{dist}^2(s^k,U^*)
		\le \inf_{u^*\in U^*}\Phi_k(u^*).
	\end{equation*}
\end{remark}

\section{Numerical Experiments}
\label{sec:experiments}

We evaluate PF-AEG and AEG-VR on a synthetic generalized-smooth VI and
three learning problems: logistic regression, phase retrieval, and
distributionally robust optimization (DRO). The synthetic example is
constructed to reflect the $\alpha$-symmetric smoothness setting,
whereas the learning problems are used to examine the empirical behavior of
the proposed adaptive mechanisms on finite-sum problems.

For deterministic comparisons, we consider GD, clipped GD \citep{zhang2019clipping}, $\beta$-normalized GD~\citep{chen2023generalized},  the adaptive golden ratio algorithm (aGRAAL) in \citep{malitsky2018graal}, the adaptive stepsize extragradient method (ads-EG), the adaptive
stepsize Popov method (ads-Popov),  the clipped stepsize extragradient
method (Clipped-EG), and the backtracking extragradient method (EG+Backtracking) in \citep{vankov2024gsvimethods}, as well as Novel-EG
\citep{choudhury2025elllipschitz}. 
 For stochastic comparisons, we use SGD, normalized SGD \citep{chen2023generalized},
 normalized SGD with momentum 	\citep{jin2021non}, clipped SGD \citep{zhang2019clipping}, SPIDER
 \citep{fang2018spider,chen2023generalized},  and the stochastic clipped extragradient method (Clipped-SEG)  in~\citep{vankov2025gssvi}. 
 
For PF-AEG, we use the reciprocal of the operator norm (or gradient norm) as the initial stepsize, whereas AEG-VR uses the squared reciprocal as the initial stepsize, with $w=1.05\,\frac{1+\beta}{1-\beta-\beta^2}\,\bar{\lambda}^{1/2}$. All other methods are run with their default settings.

\subsection{Synthetic $\alpha$-Symmetric Generalized-Smooth VI}
\label{ex:2n_non_lipschitz_vi}

We first consider the unconstrained VI on $\mathbb R^{2n}$: find
$u^*\in\mathbb R^{2n}$ such that
\begin{equation}\label{eq:synthetic-vi}
	\langle F(u^*),u-u^*\rangle\ge0,
	\qquad
	\forall u\in\mathbb R^{2n},
\end{equation}
where $u=(u_1,\ldots,u_{2n})^\top$ and
\begin{equation}\label{eq:synthetic-operator}
	F(u)=
	\begin{pmatrix}
		\operatorname{sign}(u_1)|u_1|^{p-1}+u_2\\
		\operatorname{sign}(u_2)|u_2|^{p-1}-u_1\\
		\vdots\\
		\operatorname{sign}(u_{2n-1})|u_{2n-1}|^{p-1}+u_{2n}\\
		\operatorname{sign}(u_{2n})|u_{2n}|^{p-1}-u_{2n-1}
	\end{pmatrix},
	\qquad p>2.
\end{equation}
This construction is a blockwise extension of the two-dimensional example in
\citet{vankov2024gsvimethods}. The operator is
$\alpha$-symmetric generalized smooth with $\alpha=\frac{p-2}{p-1}$, and its unique solution is $u^*=0$. Hence $\operatorname{dist}(u,U^*)=\|u\|$.

We set $n=500$ and initialize all methods at
$u_0=(1,\ldots,1)^\top\in\mathbb R^{2n}$. We consider
$p\in\{2.5,3,4\}$, corresponding to
$\alpha\in\{1/3,1/2,2/3\}$, respectively.

Figure~\ref{fig:synthetic-deterministic} compares PF-AEG with ads-EG,
ads-Popov, EG+Backtracking, Clipped-EG, and Novel-EG. Across all three tested values of $p$, PF-AEG clearly outperforms the compared deterministic methods, exhibiting a substantially faster decrease in the distance to the solution set.
\begin{figure}[ht]
	\centering

	\begin{subfigure}[t]{0.32\textwidth}
		\centering
		\includegraphics[width=\textwidth]{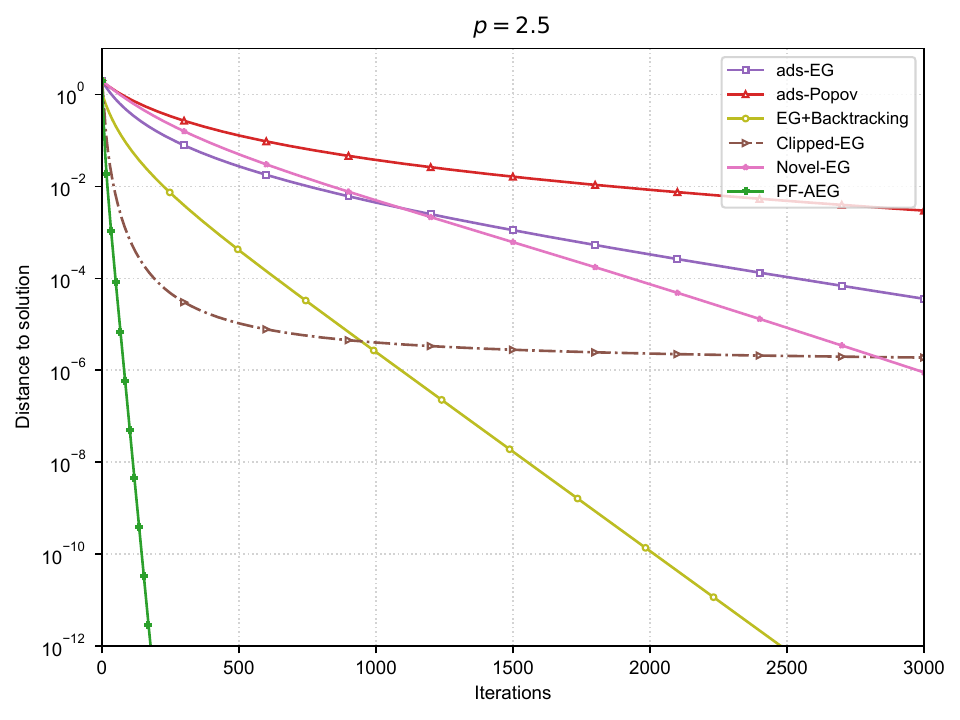}
		\caption{$p=2.5$}
		\label{fig:det-p25}
	\end{subfigure}
	\hfill
	\begin{subfigure}[t]{0.32\textwidth}
		\centering
		\includegraphics[width=\textwidth]{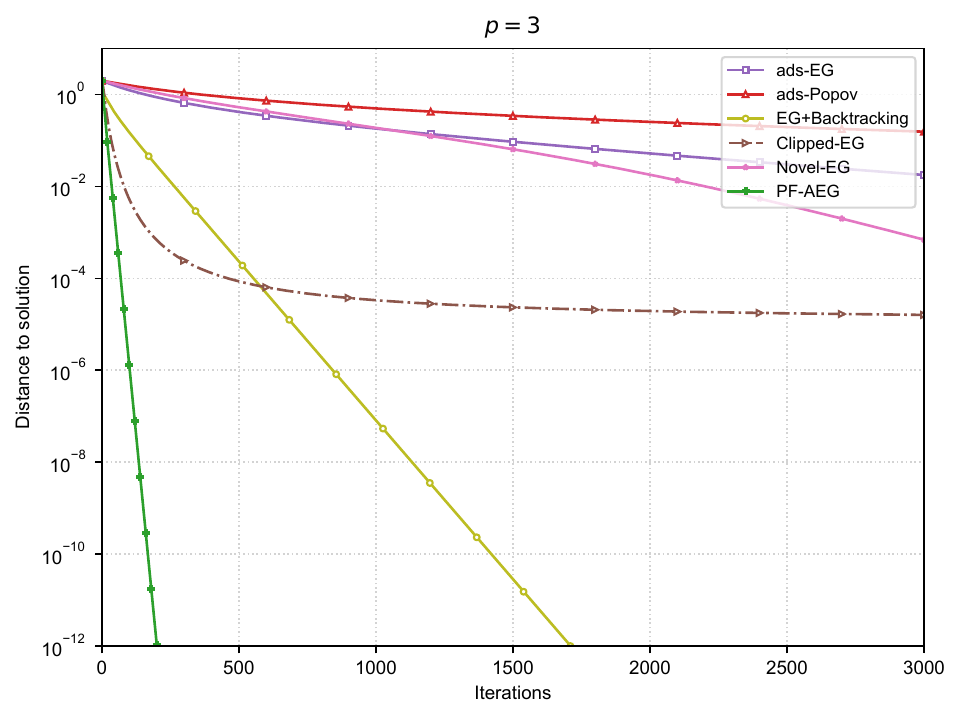}
		\caption{$p=3$}
		\label{fig:det-p3}
	\end{subfigure}
	\hfill
	\begin{subfigure}[t]{0.32\textwidth}
		\centering
		\includegraphics[width=\textwidth]{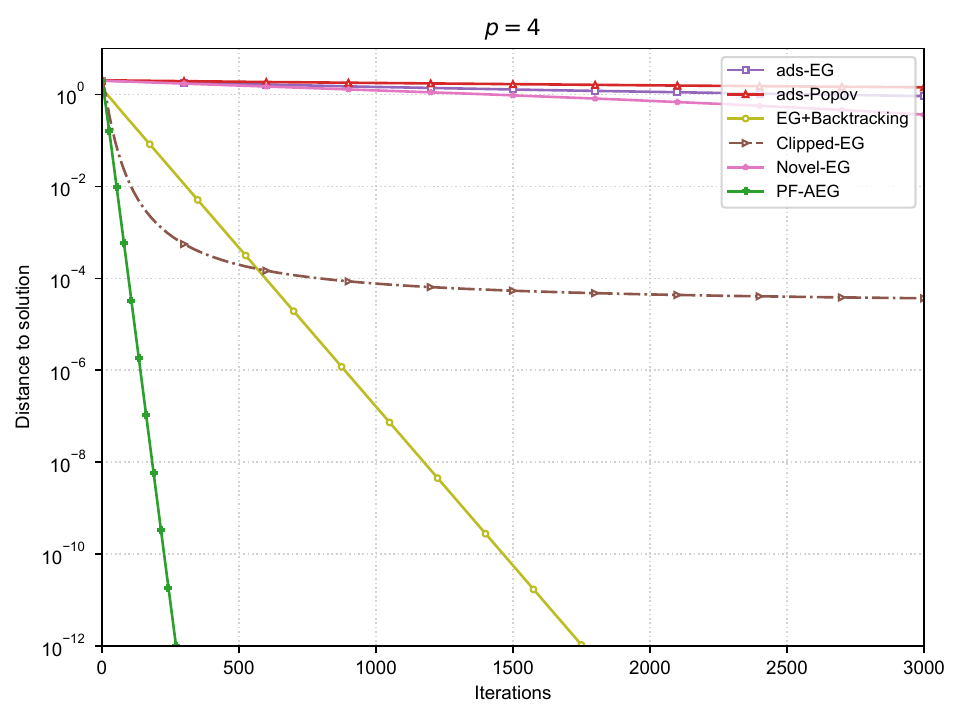}
		\caption{$p=4$}
		\label{fig:det-p4}
	\end{subfigure}

	\caption{Deterministic comparison on the synthetic
		$\alpha$-symmetric generalized-smooth VI for different values of $p$.}
	\label{fig:synthetic-deterministic}
\end{figure}

We next examine the influence of the mixing parameter $\beta$ in AEG-VR.
As shown in Figure~\ref{fig:beta-comparison}, all tested values lead to a
decrease in the reported error, while larger values of $\beta$ produce
faster progress in these runs. We therefore use $\beta=0.49$ in the
remaining AEG-VR experiments, which is consistent with the theoretical
requirement $\beta\in(0,1/2)$.

\begin{figure}[ht]
	\centering

	\begin{subfigure}[t]{0.32\textwidth}
		\centering
		\includegraphics[width=\textwidth]{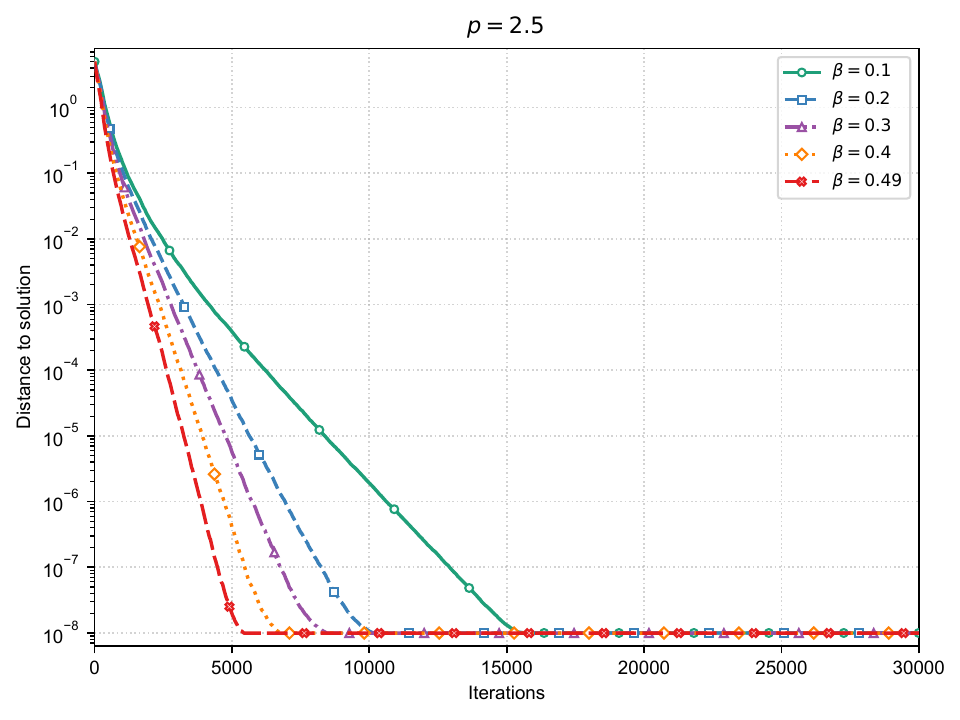}
		\caption{$p=2.5$}
		\label{fig:beta-p25}
	\end{subfigure}
	\hfill
	\begin{subfigure}[t]{0.32\textwidth}
		\centering
		\includegraphics[width=\textwidth]{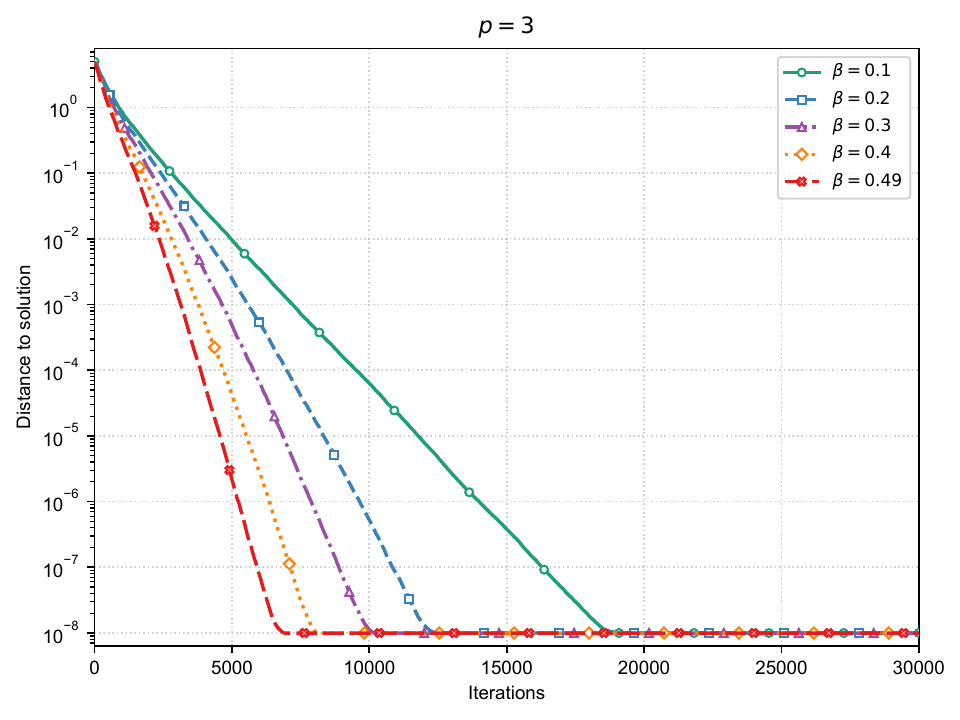}
		\caption{$p=3$}
		\label{fig:beta-p3}
	\end{subfigure}
	\hfill
	\begin{subfigure}[t]{0.32\textwidth}
		\centering
		\includegraphics[width=\textwidth]{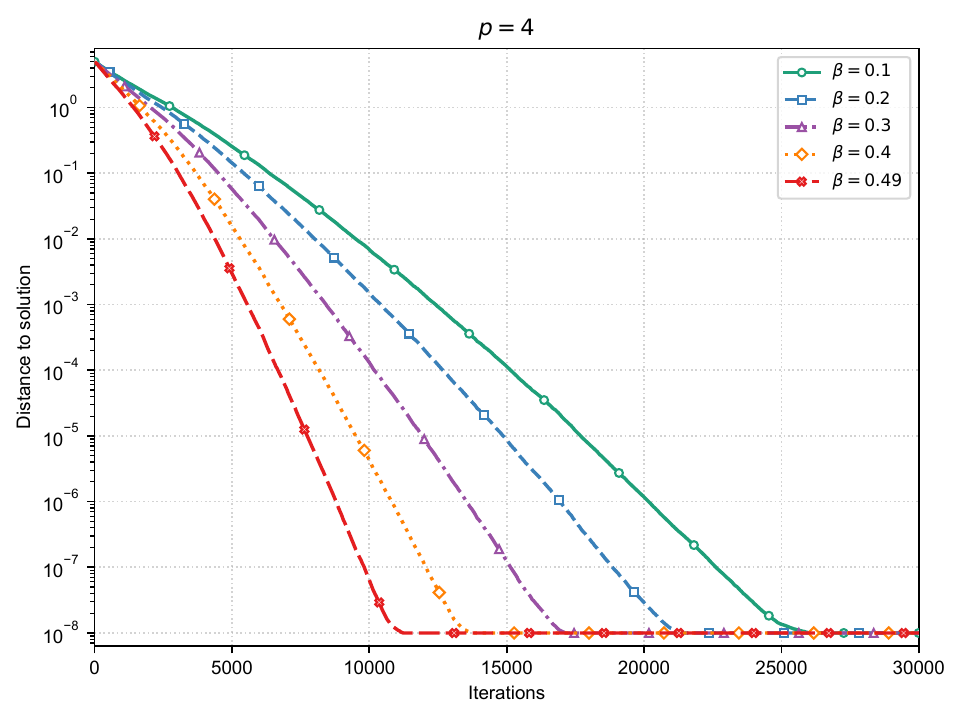}
		\caption{$p=4$}
		\label{fig:beta-p4}
	\end{subfigure}

	\caption{Effect of the mixing parameter $\beta$ on AEG-VR for the
		synthetic generalized-smooth VI for different values of $p$.}
	\label{fig:beta-comparison}
\end{figure}

Figure~\ref{fig:synthetic-stochastic} compares AEG-VR with the clipped
stochastic projected-gradient method (Clipped-SPGM) and the clipped stochastic
extragradient method (Clipped-SEG) \citep{vankov2025gssvi}. On the three
tested instances, AEG-VR yields the fastest decrease of the reported error. We exclude Clipped-SPGM from the following comparisons because it was prohibitively slow in our experiments.

\begin{figure}[ht]
	\centering

	\begin{subfigure}[t]{0.32\textwidth}
		\centering
		\includegraphics[width=\textwidth]{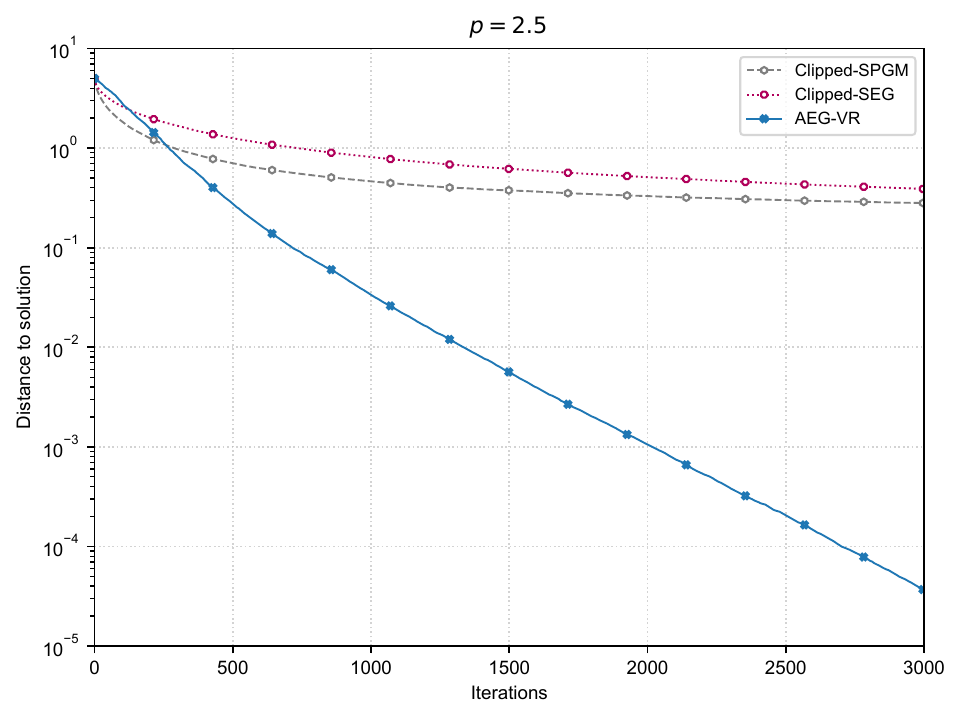}
		\caption{$p=2.5$}
		\label{fig:stochastic-p25}
	\end{subfigure}
	\hfill
	\begin{subfigure}[t]{0.32\textwidth}
		\centering
		\includegraphics[width=\textwidth]{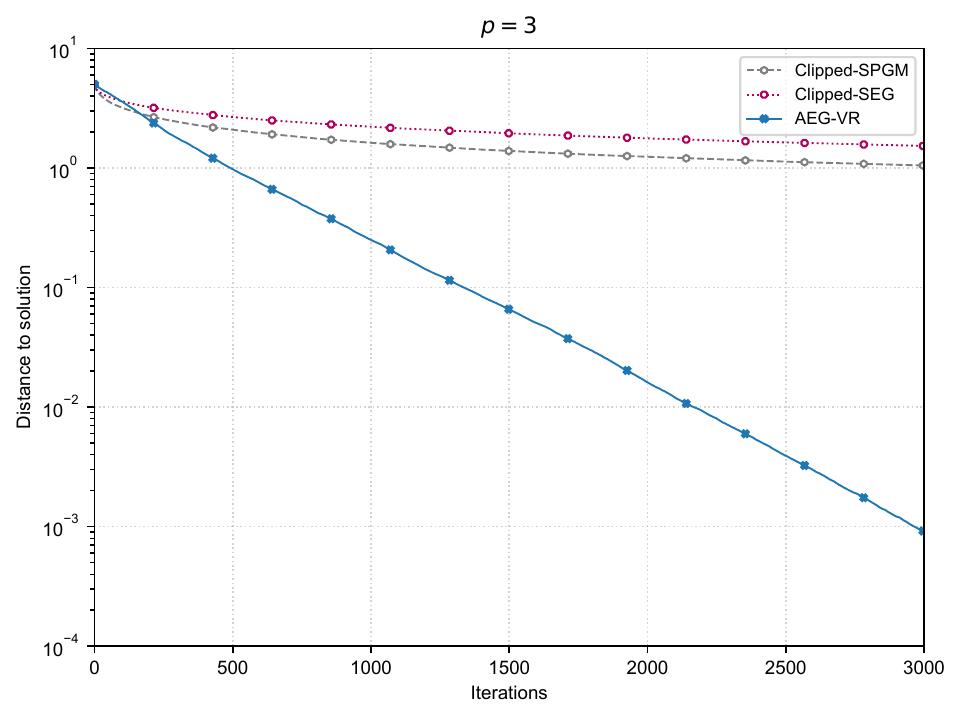}
		\caption{$p=3$}
		\label{fig:stochastic-p3}
	\end{subfigure}
	\hfill
	\begin{subfigure}[t]{0.32\textwidth}
		\centering
		\includegraphics[width=\textwidth]{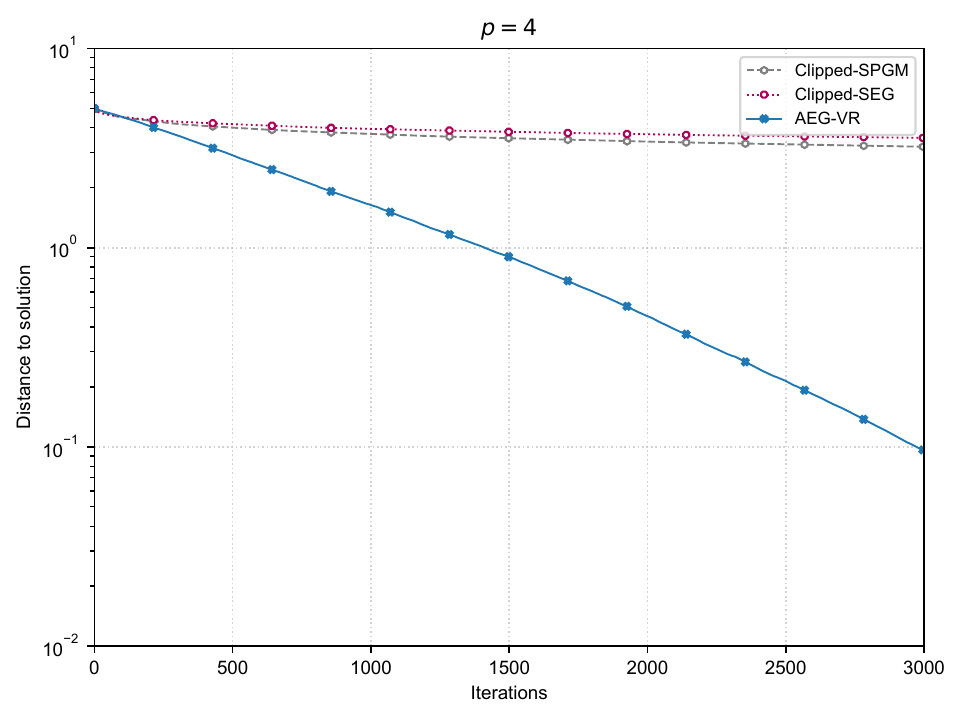}
		\caption{$p=4$}
		\label{fig:stochastic-p4}
	\end{subfigure}

	\caption{Stochastic comparison on the synthetic
		$\alpha$-symmetric generalized-smooth VI for different values of $p$.}
	\label{fig:synthetic-stochastic}
\end{figure}

\subsection{Logistic Regression}
\label{ex:logistic_regression}

We next consider logistic regression,
\begin{equation}\label{lrp}
	\min_{x\in\mathbb R^d}
	f(x):=\frac1N\sum_{i=1}^{N}
	\log\!\left(1+\exp(-y_i a_i^\top x)\right),
\end{equation}
where $a_i\in\mathbb R^d$ and $y_i\in\{-1,1\}$.
The experiments use the \texttt{a9a} and \texttt{mushrooms} datasets from
the LIBSVM collection. A high-accuracy reference solution $x^*$ is computed
using L-BFGS-B \citep{Byrd1995} on the full  objective.

For the deterministic comparison, PF-AEG is compared with clipped GD,
normalized GD, aGRAAL, ads-EG, ads-Popov, EG+Backtracking, Clipped-EG, and
Novel-EG. Figures~\ref{fig:logreg-4panels}(a)--(b) show that, PF-AEG reaches the target accuracy in fewer iterations than other methods on both datasets.

For the stochastic comparison, AEG-VR is evaluated against SGD, normalized
SGD, normalized SGD with momentum, clipped SGD, SPIDER, and Clipped-SEG.
Figures~\ref{fig:logreg-4panels}(c)--(d) show that, 
AEG-VR decreases the error $f(x_k)-f(x^*)$ more rapidly than the compared methods.

\begin{figure}[ht]
	\centering
	\begin{minipage}[t]{0.48\linewidth}
		\centering
		\includegraphics[width=\linewidth]{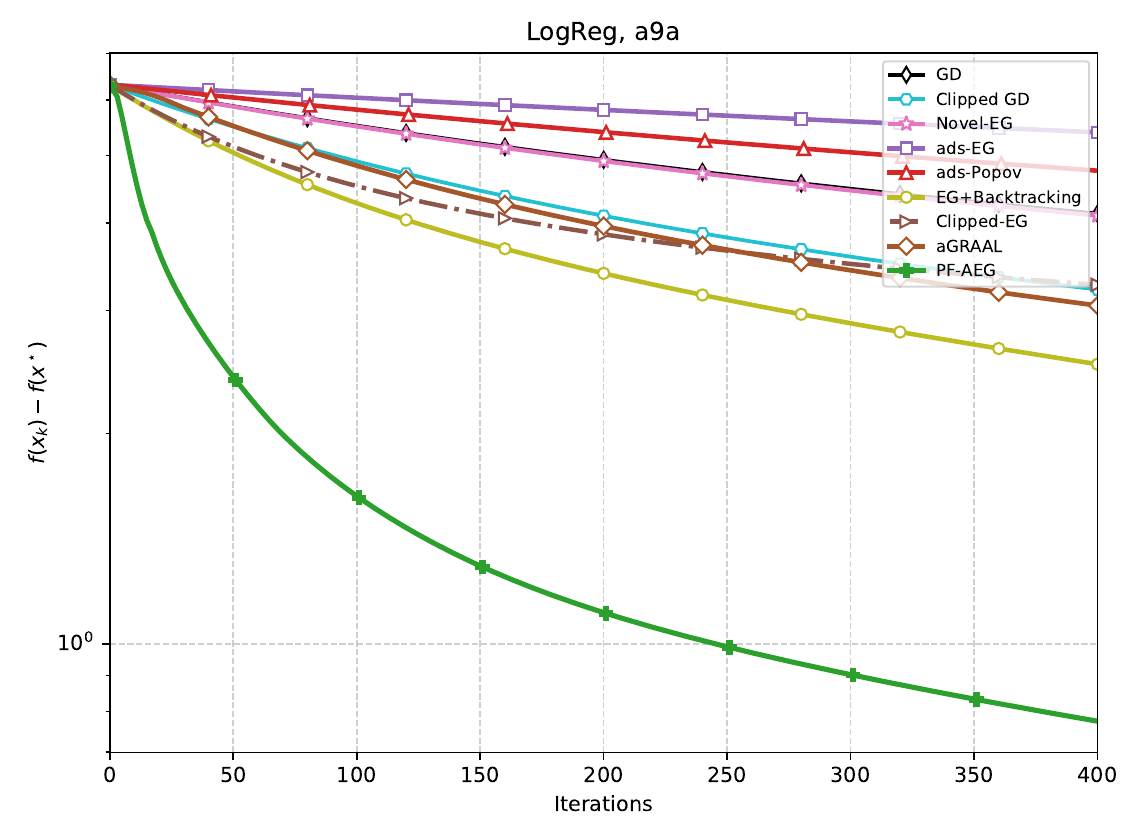}\\[-0.5ex]
		\small (a) \texttt{a9a}, deterministic
	\end{minipage}\hfill
	\begin{minipage}[t]{0.48\linewidth}
		\centering
		\includegraphics[width=\linewidth]{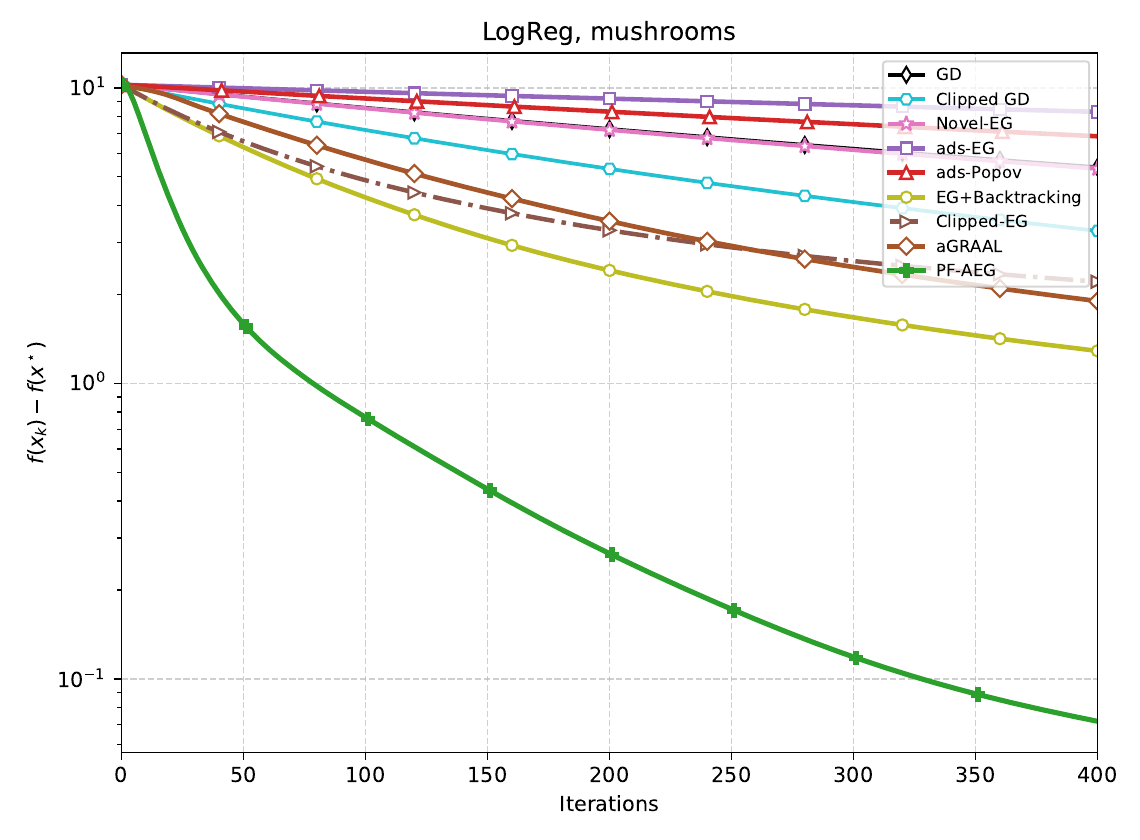}\\[-0.5ex]
		\small (b) \texttt{mushrooms}, deterministic
	\end{minipage}
	
	\vspace{0.4em}
	
	\begin{minipage}[t]{0.48\linewidth}
		\centering
		\includegraphics[width=\linewidth]{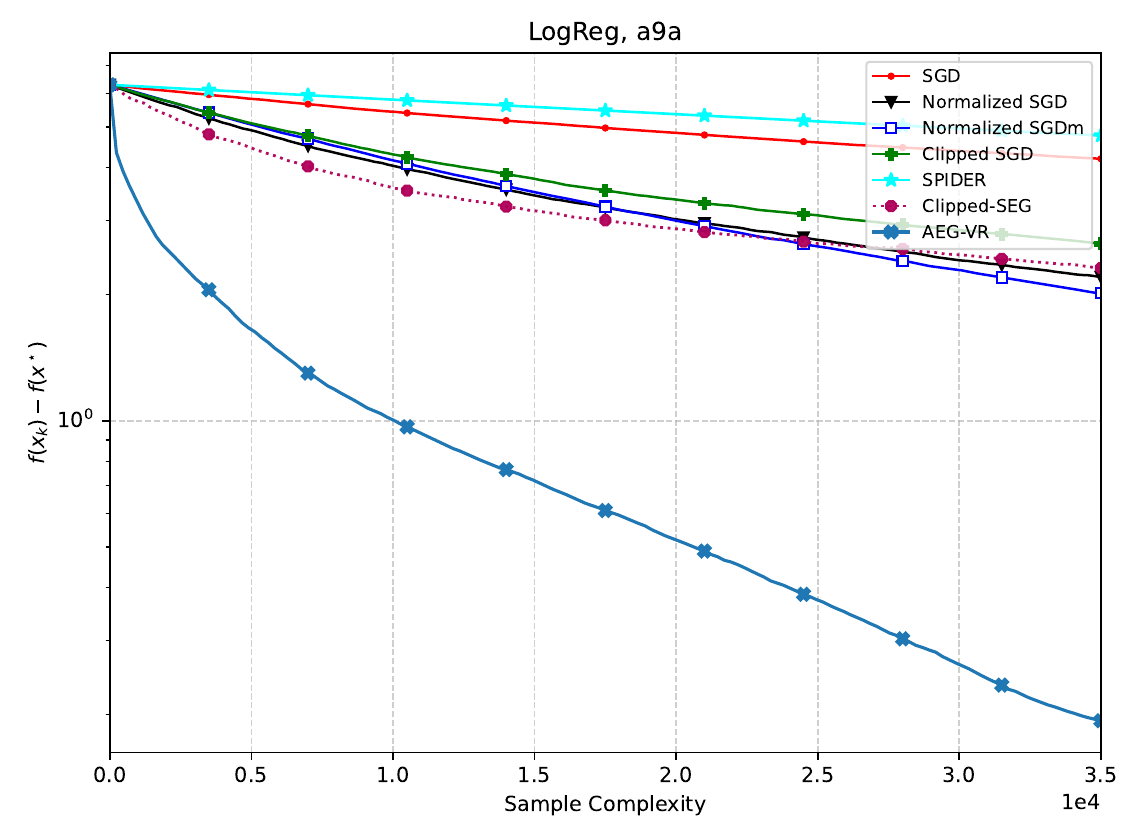}\\[-0.5ex]
		\small (c) \texttt{a9a}, stochastic
	\end{minipage}\hfill
	\begin{minipage}[t]{0.48\linewidth}
		\centering
		\includegraphics[width=\linewidth]{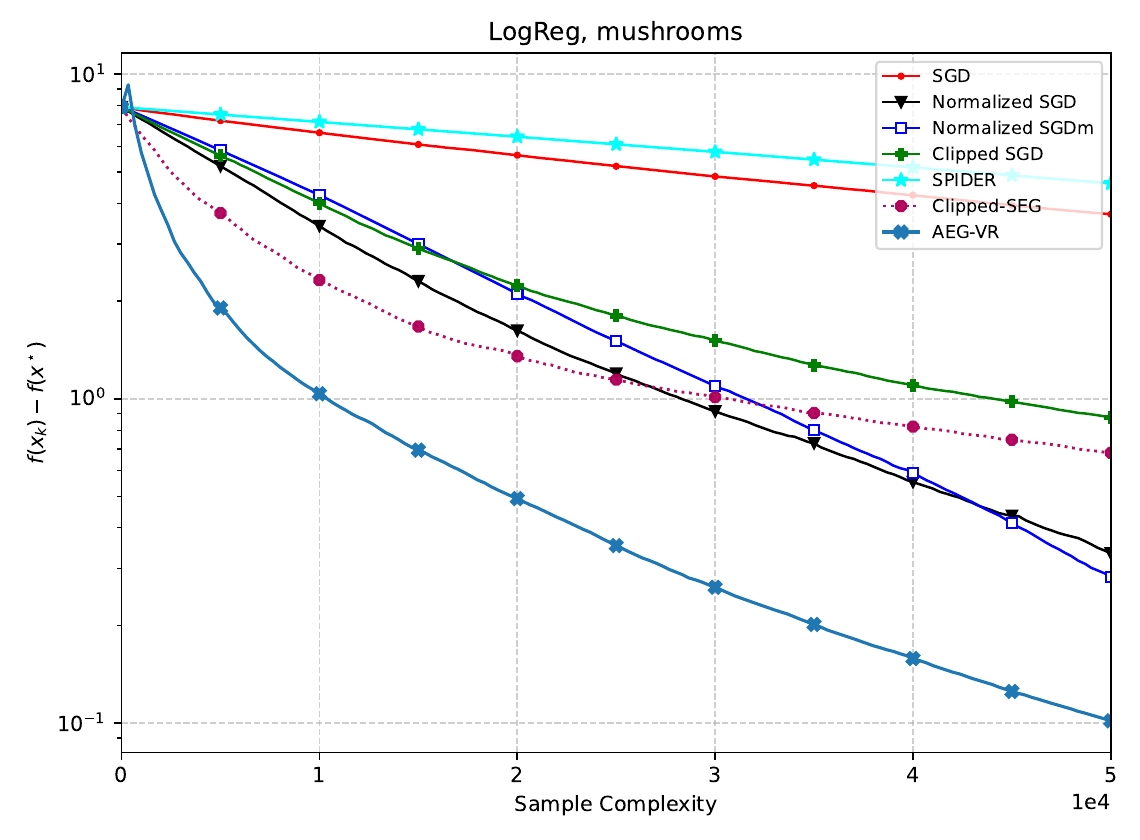}\\[-0.5ex]
		\small (d) \texttt{mushrooms}, stochastic
	\end{minipage}
	\caption{Deterministic and stochastic comparisons for logistic regression
		on the \texttt{a9a} and \texttt{mushrooms} datasets.}
	\label{fig:logreg-4panels}
\end{figure}

\subsection{Phase Retrieval}
\label{ex:phase-retrieval}

We consider the phase-retrieval problem
\begin{equation}\label{pr}
	\min_{z\in\mathbb R^d}
	f(z):=\frac{1}{2m}\sum_{r=1}^{m}
	\left(y_r-|a_r^\top z|^2\right)^2,
\end{equation}
where $a_r\in\mathbb R^d$ is the $r$-th measurement vector.
We set $d=100$ and $m=3000$. The entries of $a_r$ and the
ground-truth signal $z^\star$ are generated independently from
$\mathcal N(0,0.5)$, while the initial point $z_0$ is generated
entrywise from $\mathcal N(5,0.5)$. The observations are
\begin{equation*}
	y_r=|a_r^\top z^\star|^2+n_r,
	\qquad
	n_r\sim\mathcal N(0,4^2).
\end{equation*}

Figure~\ref{fig:phase-retrieval}(a) gives the deterministic comparison.
PF-AEG decreases the objective function value rapidly and outperforms  the
best-performing baseline in this experiment. Figure~\ref{fig:phase-retrieval}(b)
gives the stochastic comparison, where AEG-VR remains competitive with the
strongest stochastic baselines while using no generalized-smoothness or quasi-sharpness constants
in its stepsize rule.

\begin{figure}[ht]
	\centering
	\begin{minipage}[t]{0.47\linewidth}
		\centering
		\includegraphics[width=\linewidth]
		{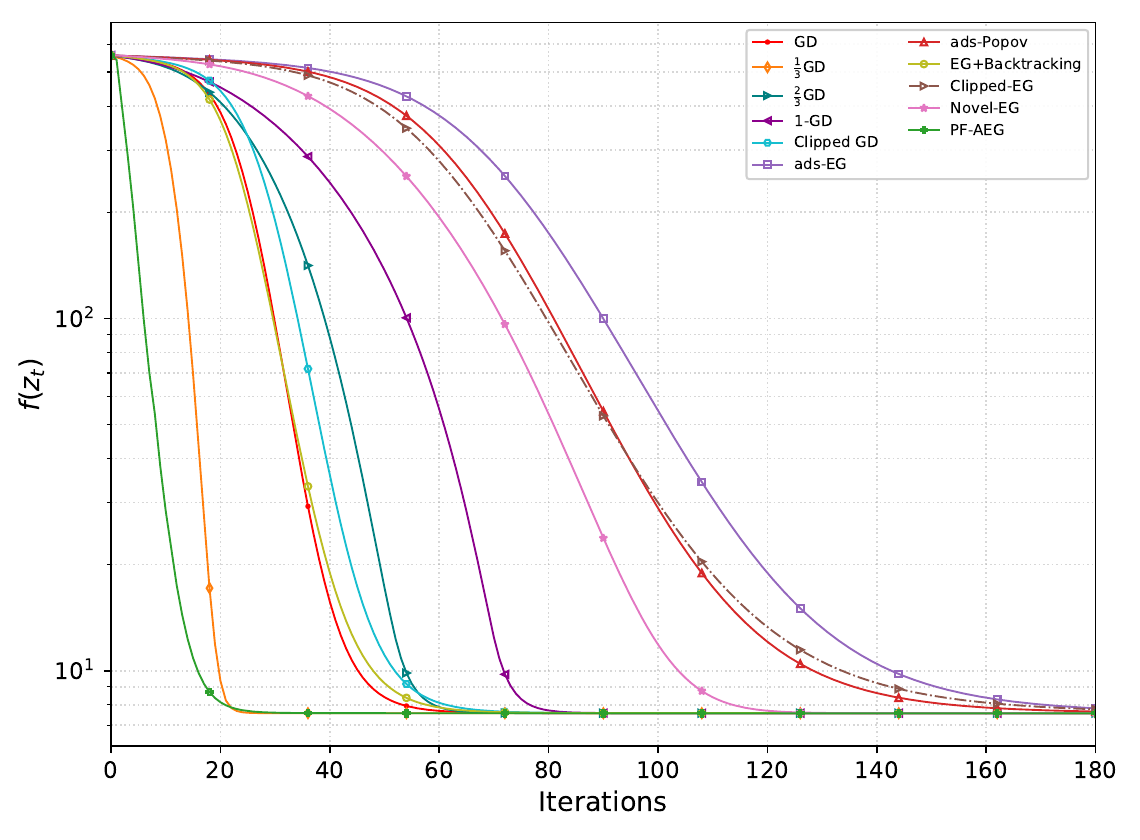}\\[-0.5ex]
		\small (a) Deterministic methods
	\end{minipage}
	\hfill
	\begin{minipage}[t]{0.47\linewidth}
		\centering
		\includegraphics[width=\linewidth]
		{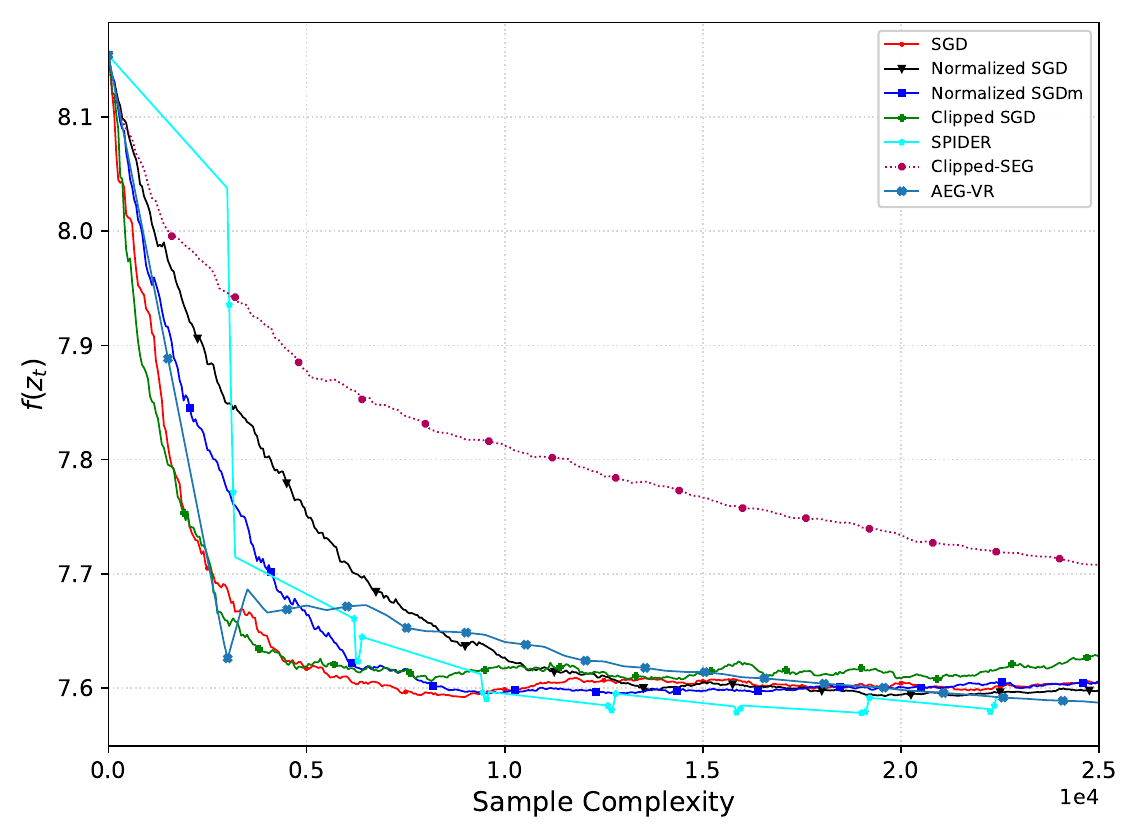}\\[-0.5ex]
		\small (b) Stochastic methods
	\end{minipage}
	\caption{Deterministic and stochastic comparisons for the
		phase-retrieval problem \eqref{pr}.}
	\label{fig:phase-retrieval}
\end{figure}

\subsection{Distributionally Robust Optimization}
\label{ex:dro}

We finally consider the DRO problem obtained from the standard dual
reformulation with a $\psi$-divergence penalty:
\begin{equation}\label{eq:dro-objective}
	\min_{\theta\in\mathbb R^d,\,\eta\in\mathbb R}
	L_N(\theta,\eta):=
	\frac{\tau}{N}\sum_{i=1}^{N}
	\psi^*\!\left(\frac{\ell_i(\theta)-\eta}{\tau}\right)+\eta,
\end{equation}
where $\tau>0$ is the penalty parameter and $\psi^*$ is the convex
conjugate associated with the chosen divergence.

We use the Life Expectancy dataset and retain the first $N=2000$
preprocessed observations $\{(x_i,y_i)\}_{i=1}^{N}$, where
$x_i\in\mathbb R^{34}$ and $y_i\in\mathbb R$.
Missing entries are imputed using featurewise medians. Numerical features are
winsorized and standardized, and the categorical variables
\texttt{country} and \texttt{status} are removed. Independent Gaussian
perturbations $\epsilon_i\sim\mathcal N(0,1)$
are added to the target values.

We set $\tau=0.01$ and use
\begin{equation*}
	\psi^*(t)=\frac{(t+2)_+^2}{4}-1,
\end{equation*}
corresponding to the $\chi^2$-divergence. For sample $i$, the loss is
\begin{equation*}\label{eq:dro-loss}
	\ell_i(\theta)=
	\frac12\left(y_i-x_i^\top\theta\right)^2
	+0.1\sum_{j=1}^{34}
	\log\!\left(1+|\theta^{(j)}|\right).
\end{equation*}
We report the profile objective
\begin{equation}\label{eq:dro-profile}
	\Psi(\theta_k):=\min_{\eta\in\mathbb R}L_N(\theta_k,\eta).
\end{equation}

Figure~\ref{fig:dro-exp}(a) presents the deterministic comparison. On the reported instance, PF-AEG decreases the profile objective more rapidly than the comparison methods after about 10 iterations. Figure~\ref{fig:dro-exp}(b) shows the stochastic comparison, in which AEG-VR outperforms the other stochastic methods.

\begin{figure}[ht]
	\centering
	\begin{minipage}[t]{0.48\linewidth}
		\centering
		\includegraphics[width=\linewidth]
		{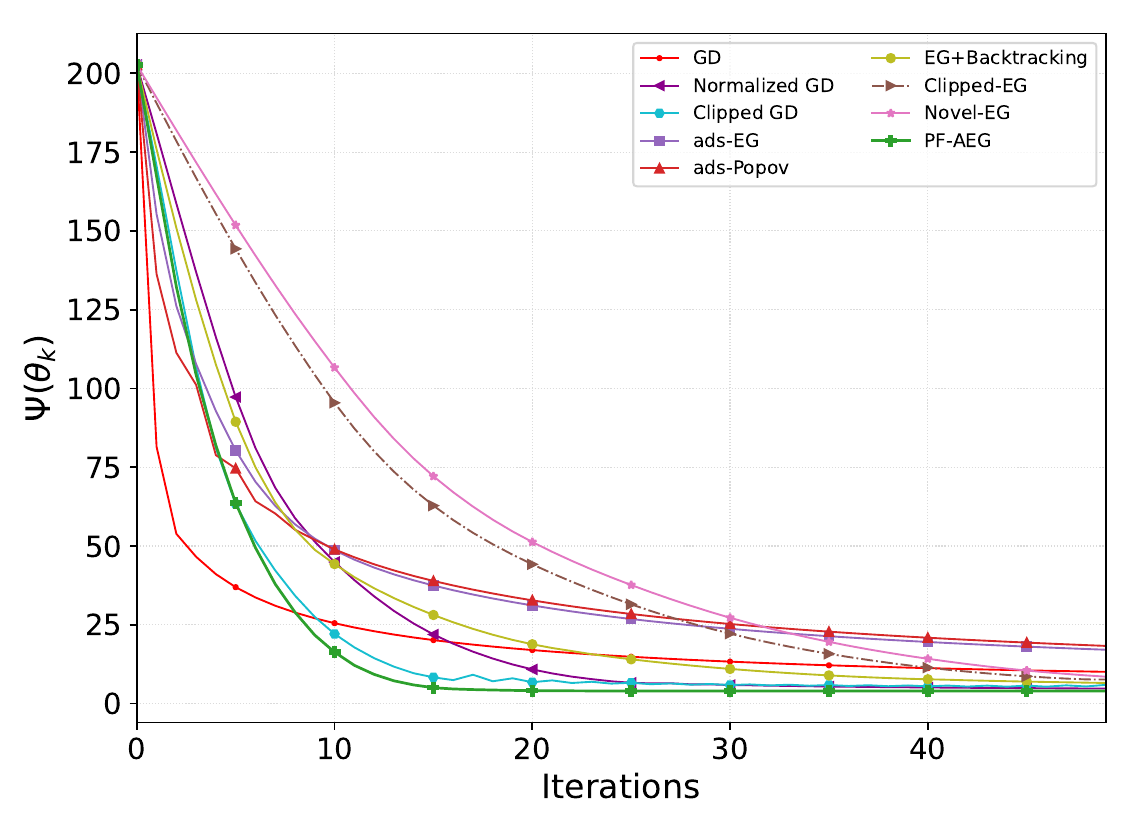}\\[-0.5ex]
		\small (a) Deterministic methods
	\end{minipage}
	\hfill
	\begin{minipage}[t]{0.48\linewidth}
		\centering
		\includegraphics[width=\linewidth]
		{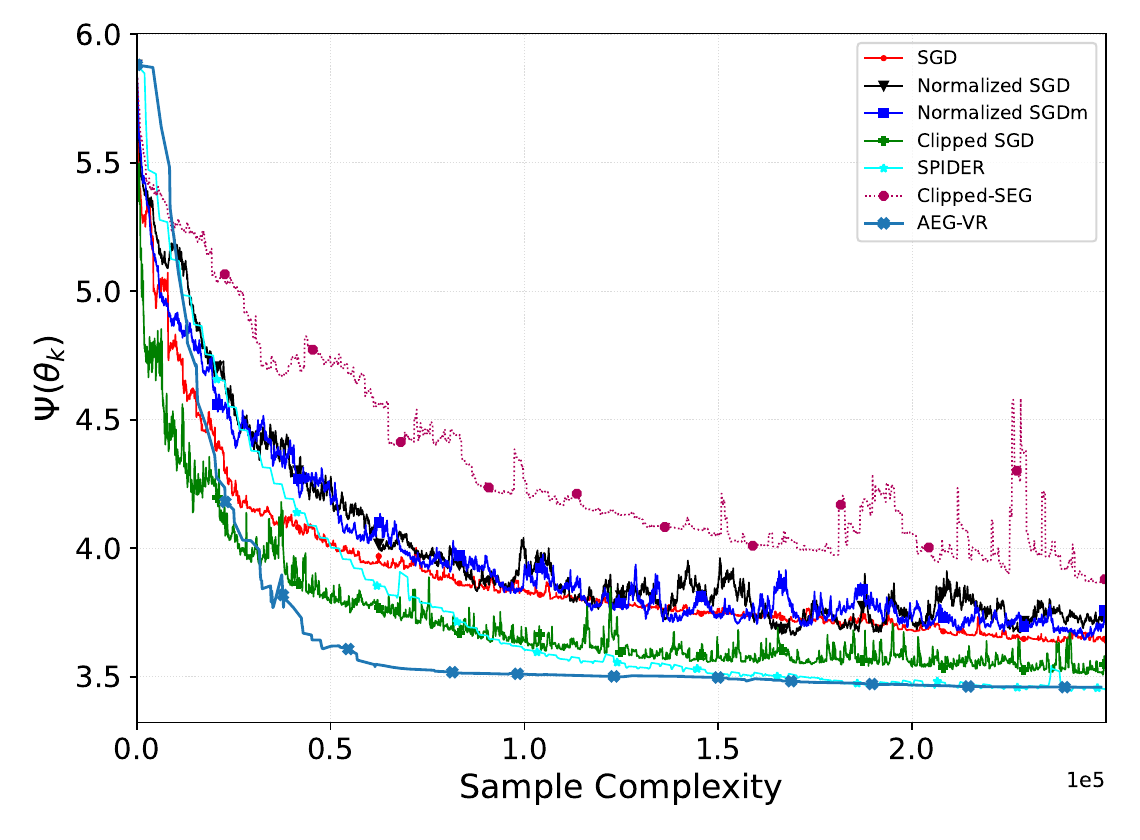}\\[-0.5ex]
		\small (b) Stochastic methods
	\end{minipage}
	\caption{Deterministic and finite-sum comparisons for the DRO experiment
		on the Life Expectancy dataset.}
	\label{fig:dro-exp}
\end{figure}

Overall, the experiments illustrate the empirical behavior of PF-AEG and AEG-VR in settings ranging from a synthetic generalized-smooth VI to finite-sum learning problems. PF-AEG often outperforms  the adaptive, clipped, normalized, and backtracking methods in the deterministic
experiments, while AEG-VR performs favorably against the stochastic and
variance-reduced baselines in the reported finite-sum experiments. These
results are consistent with the intended role of the proposed stepsize rules:
they adapt using quantities generated by the iterates without requiring the
generalized-smoothness or quasi-sharpness parameters.

\section{Conclusions}
\label{sec:conclusion}

We introduced adaptive extragradient methods for variational inequalities
under $\alpha$-symmetric generalized smoothness and
$p$-quasi-sharpness. We developed PF-AEG, an explicit and
line-search-free method whose stepsizes are determined entirely by quantities
generated during the iterations and therefore require no prior knowledge of
the generalized-smoothness or quasi-sharpness parameters. Its last-iterate
analysis yields linear convergence for the full range $0<p\le2$, extending
the corresponding guarantee beyond the previously studied case $p=2$, and
recovers the rate $\mathcal O(k^{-2/(p-2)})$ for $p>2$. Building on the
same adaptive principle, AEG-VR extends the construction to finite-sum VIs by
combining adaptive stepsizes with variance reduction and randomized snapshot
updates. The numerical experiments on synthetic generalized-smooth VIs and
learning problems show that the proposed methods compare favorably with the
 adaptive, clipped, normalized, backtracking, and stochastic
baselines, while avoiding problem-dependent parameters in stepsize selection.


\vskip 0.2in
\bibliographystyle{plainnat}

\end{document}